\documentclass[reqno,twoside,11pt]{amsart}
\usepackage{amsfonts,amsmath,amssymb}
\usepackage{mathrsfs,mathtools}
\usepackage{hyperref}
\usepackage{esint}
\usepackage{bm}
\usepackage{cite}
\usepackage{enumitem}
\usepackage{dsfont}

\usepackage{tikz}
\usetikzlibrary{matrix,arrows.meta}

\usepackage[dvips,bottom=1.4in,right=1in,top=1in, left=1in]{geometry}

\DeclareMathAlphabet{\mathpzc}{OT1}{pzc}{m}{it}

\newcommand{\dd}{\mathrm{d}}
\newcommand{\tr}{\mathrm{tr}_\Omega}
\newcommand{\Tr}{\mathrm{Tr}_\Omega\,}
\newcommand{\TR}{\mathrm{TR}_\Omega\,}

\def\dif{\:\mathrm{d}}

\def\0{\emptyset}

\def\tr{\mathrm{tr}_\Omega\:}

\def\dif{\:\mathrm{d}}

\def\dif{\:\mathrm{d}}

\newtheorem{theorem}{Theorem}[section]

\newtheorem{definition}[theorem]{Definition}
\newtheorem{example}[theorem]{Example}
\newtheorem{lemma}[theorem]{Lemma}

\newtheorem{proposition}[theorem]{Proposition}

\newtheorem{remark}[theorem]{Remark}
\numberwithin{equation}{section}

\usepackage{soul}
\begin{document}

\title{The  Spatially Variant Fractional Laplacian}

\thanks{
{CNR was partially supported by NSF grant 2012391. CNR was also been supported via the framework of MATHEON by the Einstein Foundation Berlin within the ECMath project SE15/SE19 and of the Excellence Cluster Math+ Berlin within project AA4-3 and the transition project OT6, and acknowledges the support of the DFG through the DFG-SPP 1962: Priority Programme ``Non-smooth and Complementarity-based Distributed Parameter Systems: Simulation and Hierarchical Optimization''  within Project 11. ANC was partially supported by Air Force Office of Scientific Research under Award NO: FA9550-19-1-0036, by CONICET grant PIP 0032CO, and by ANPCyT grant PICT 2016-1022.}
}

\thanks{The authors are very grateful to Prof. Harbir Antil from George Mason University for the valuable discussions on the paper.}


\author{Andrea N. Ceretani}
\address{Andrea N. Ceretani. Department of Mathematics of the Faculty of Exact and Natural Sciences, University of Buenos Aires, and Mathematics Research Institute ``Luis A. Santal\'o'' (IMAS), CONICET, Argentina.}
\email{aceretani@dm.uba.ar}

\author{Carlos N. Rautenberg}
\address{Carlos N. Rautenberg. Department of Mathematical Sciences and the Center for Mathematics and Artificial Intelligence (CMAI), George Mason University, Fairfax, VA 22030, USA.}
\email{crautenb@gmu.edu}

\begin{abstract}
We introduce a definition of the fractional Laplacian $(-\Delta)^{s(\cdot)}$ with spatially variable order $s:\Omega\to [0,1]$ and study the solvability of the associated Poisson problem on a bounded domain $\Omega$. The initial motivation arises from the extension results of Caffarelli and Silvestre, and Stinga and Torrea; however the analytical tools and approaches developed here are new. For instance, in some cases we allow the variable order $s(\cdot)$ to attain the values $0$ and $1$ leading to a framework on weighted Sobolev spaces with non-Muckenhoupt weights.  
Initially, and under minimal assumptions, the operator $(-\Delta)^{s(\cdot)}$ is identified as the Lagrange multiplier corresponding to an optimization problem; and its domain is determined as a quotient space of weighted Sobolev spaces. The well-posedness of the associated Poisson problem is then obtained for data in the dual of this quotient space. Subsequently, two trace regularity results are established, allowing to partially characterize functions in the aforementioned quotient space whenever a Poincar\'e type inequality is available. Precise examples are provided where such inequality holds, and in this case the domain of the operator $(-\Delta)^{s(\cdot)}$ is identified with a subset of a weighted Sobolev space with spatially variant smoothness $s(\cdot)$. The latter further allows to prove the well-posedness of the Poisson problem assuming functional regularity of the data.
\end{abstract}


 
 \maketitle

\section{Introduction}
The goal of this work is twofold: (i) introduce the spectral fractional Laplacian $(-\Delta)^{s(\cdot)}$ {associated with a homogeneous Dirichlet condition} on a bounded domain $\Omega\subset\mathbb{R}^N$, $N\geq 1$, in the case the fractional order $s(\cdot)$ is spatially variable and possibly attains the values $0$ and $1$; (ii) study the well-posedness of the equation 
\begin{equation}\label{Eq:IntroElliptic}
\begin{split}
(-\Delta)^{s(\cdot)}v&=h \quad \text{ in } \Omega,\\
 v& = 0 \quad  \text{ on } \partial \Omega,
\end{split}
\end{equation}
for some classes of data $h$, and where $v=0$ is understood in an appropriate sense.

Motivated by the extension approach in $\mathbb{R}^N$ by Caffarelli and Silvestre \cite{CaSi2007}, or in bounded domains by Stinga and Torrea \cite{StTo2010}, we define $(-\Delta)^{s(\cdot)}$ to be the Lagrange multiplier associated to a suitable {variational problem defined in an extended domain}, for measurable functions $s(\cdot)$ with range contained in the interval $[0,1]$. For a general class of functions $s(\cdot)$, the domain of $(-\Delta)^{s(\cdot)}$ can be identified with a quotient space $\mathscr{X}(\Omega,w)$ involving weighted Sobolev spaces,
\begin{equation}\label{Eq:IntroDomain}
\mathscr{X}(\Omega,w):=\mathscr{L}_{0,L}^{1,2}(\mathcal{C},w)/\mathscr{L}_0^{1,2}(\mathcal{C},w),
\end{equation}
where $\mathcal{C}=\Omega\times (0,+\infty)$ is the open semi-infinite cylinder {(the extended domain)} with base $\Omega$, and $w$ is a specific weight function. Roughly speaking, the spaces $\mathscr{L}_{0,L}^{1,2}(\mathcal{C},w)$ and $\mathscr{L}_0^{1,2}(\mathcal{C},w)$ are composed {of} functions that vanish on the lateral boundary of $\mathcal{C}$, and  on the whole boundary {(including the base $\Omega$)}, respectively. Equation \eqref{Eq:IntroElliptic} is then solvable for every $h$ in the dual space of $\mathscr{X}(\Omega,w)$. For a smaller class of $s(\cdot)$, the domain can be identified as a subset of a weighted Lebesgue space $L^2(\Omega, \tilde{w})$ for some function $\tilde{w}$, and the equation \eqref{Eq:IntroElliptic} is  solvable when the right hand side is in  $L^2(\Omega, \tilde{w})$. For an even smaller class of functions $s(\cdot)$, this result is {further} improved since the domain of $(-\Delta)^{s(\cdot)}$ is identified with a subset of a weighted Sobolev space of functions with spatially variable smoothness, related to $s(\cdot)$.

The main application that has motivated this work, in addition to the natural theoretical interest, is the recent paper \cite{AnRa2019}. There, initial results on an extension approach in Hilbert spaces on an open cylinder with base $\Omega$ are given. However, the authors stopped short of defining $(-\Delta)^{s(\cdot)}$ due to the lack of a proper functional framework. The current paper aims to fill this gap. It is worth mentioning that none of the existing results in the literature are applicable to our case and new PDE and variational analysis tools are needed to study the current situation. For example, the extension approaches in \cite{CaSi2007,StTo2010} assume $s \in (0,1)$ to be a constant and avoid the extreme cases of $0$ and $1$. In this setting, the nonlocal problem $(-\Delta)^{s}v=h$ in $\Omega$, where $(-\Delta)^s$ is the $s$-power of the realization of $-\Delta$ in $L^2(\Omega)$ with zero Dirichlet boundary conditions, can be equivalently formulated as a local one on a Sobolev space with a Muckenhoupt weight. On the other hand, our $s(\cdot)$ is a function which is allowed to touch the extreme cases 0 and 1 and therefore, the associated weights do not fulfill the Muckenhoupt property \cite[Proposition~1]{AnRa2019}. In particular, fundamental results of type ``$H=W$" or Poincar\'e inequalities are not known in our case, leading to a more complex functional analytic framework.

The literature concerning possible definitions of $(-\Delta)^{s(\cdot)}$ with non-constant $s$ is restricted to the stochastic processes and stochastic calculus approaches and considers always the unbounded case $\Omega=\mathbb{R}^N$; see the monograph \cite{Applebaum_2009} and the references therein. By means of the L\'{e}vy-Khintchine representation formula, and the Fourier transform, the operator is determined to be of L\'{e}vy type. However, strong additional assumptions on $s(\cdot)$ are required to show that the operator is associated to a Feller or a Markov process. To name a few, these include assuming that $s(\cdot)$ is Lipschitz continuous and satisfies $\varepsilon\leq s(\cdot)\leq 1-\varepsilon$ for some $\varepsilon\in(0,1)$;  
see 
\cite[Example 3.5.9]{Applebaum_2009}. Neither of these restrictions are present in this work. 

The paper is further motivated by several applications. The extension approach with spatially varying $s(\cdot)$ has shown
remarkable potential in image denoising: A rough choice of $s(\cdot)$ performs better than an optimal selected regularization parameter in total variation approaches; see \cite{AnRa2019}. This is indeed a game changer, especially the variable $s(\cdot)$ approach can enable one to replace the nonlinear (and degenerate) Euler-Lagrange equations in case of total variation by a linear one in the case of the variable fractional.
{The variable $s(\cdot)$ approach can be also applied in geophysics: Models governed by a fractional Helmholtz equation (with constant fractional order) have shown good qualitative agreement with available magnetoteluric data, see \cite{CWeiss_BvBWaanders_HAntil_2018a}. Given the spatially long-range correlated heterogeneity of the medium, nonlocal models with spatially varying fractional order $s(\cdot)$ appear as an atractive tool to further obtain quantitative agreement.}

\medskip
\noindent
{\bf Outline}. The notation and main assumptions we make, specially those for the variable exponent $s(\cdot)$, are specified in Section \ref{Sect:Assumptions}. In Section \ref{Sect:Constant} we provide a succinct idea of the approach that we follow to study the fractional Laplacian with spatially variable order, $(-\Delta)^{s(\cdot)}$, which is motivated by well-known results for the usual spectral fractional Laplacian. 

Our main results begin from Section \ref{Sect:Abstract}, where we introduce a definition of $(-\Delta)^{s(\cdot)}$ on the quotient space $\mathscr{X}(\Omega,w)$. Also in this section we prove {the} existence and uniqueness of a solution $v\in\mathscr{X}(\Omega,w)$ to the associated Poisson problem \eqref{Eq:IntroElliptic} for every $h$ in the dual space of $\mathscr{X}(\Omega,h)$. It is worth mentioning that the results in Section \ref{Sect:Abstract} require minimal conditions on the function $s(\cdot)$, the weight $w$, and the domain $\Omega$. The results given in Section \ref{Sect:Abstract}, however, do not provide conditions for solvability of the Poisson problem when the right hand side of the elliptic equation is a (regular) real valued function defined only on $\Omega$.

In a second approach, we are able to better identify the domain of $(-\Delta)^{s(\cdot)}$  as  a quotient space also, now on a Sobolev space $\mathscr{H}_{0,L}^{1,p}(\mathcal{C},w)$ that consist of functions in $W^{1,p}(\mathcal{C},w)$ that formally vanish on the lateral boundary of $\mathcal{C}$. Differently from the construction given in Section \ref{Sect:Abstract}, this second approach requires some extra conditions on both, $s(\cdot)$ and $\Omega$. These conditions are intimately related with the existence of $\Omega$-trace results for functions in $\mathscr{H}_{0,L}^{1,2}(\mathcal{C},w)$, as well as with the existence of a Poincar\'e inequality in $\mathscr{H}_{0,L}^{1,2}(\mathcal{C},w)$;  thus, we postpone the second construction until Section \ref{Sect:Elliptic}.  

In Section \ref{Sect:Traces}, we first study the $\Omega$-traces of functions in $\mathscr{H}_{0,L}^{1,p}(\mathcal{C},w)$, for $2\leq p<\infty$. In particular, we are able to characterize $s(\cdot)$-dependent  integrability  and differential regularity of restrictions of functions in $\mathscr{H}_{0,L}^{1,p}(\mathcal{C},w)$ to $\Omega$. Subsequently, we are able to prove the existence of a Poincar\'e inequality for $\mathscr{H}_{0,L}^{1,p}(\mathcal{C},w)$ in Section \ref{Sect:PoincareIneq}, for a special class of non-constant $s(\cdot)$ functions. 

Our results finish in Section \ref{Sect:Elliptic}, where the details on the second definition of $(-\Delta)^{s(\cdot)}$ are given. Here, we identify the domain of $(-\Delta)^{s(\cdot)}$ with a subset of a weighted Lebesgue space $L^2(\Omega,\tilde{w})$ for some weight $\tilde{w}$, provided $s(\cdot)$ vanishes only on a set of zero measure and a Poincar\'e inequality holds for functions in $\mathscr{H}_{0,L}^{1,2}(\mathcal{C},w)$. Further, we improve this result for the case when $\Omega$ is the $N$-dimensional unit square and $s(\cdot)$ satisfies some extra conditions. In this latter case we identify the domain of $(-\Delta)^{s(\cdot)}$ with a subset of a Sobolev space of functions with variable smoothness on $\Omega$. The paper closes with Section \ref{Sect:Conclusions} that includes, in addition to conclusions, a number of open questions and future research directions. 

{More general elliptic operators of the form
\begin{equation*}
(-\mathrm{div} A\nabla)^{s(\cdot)}v=h,
\end{equation*} 
with spatially variable fractional order $s(\cdot)$, can be defined by extending the ideas in this paper in a natural way.}

\section{Notation and main assumptions}\label{Sect:Assumptions}
We assume that $\Omega\subset\mathbb{R}^N$, $N\geq 1$, is a non-empty bounded open set with a Lipschitz boundary {$\partial\Omega$} (except in Section \ref{Sect:Abstract}, where no condition is imposed on the $\Omega$ boundary). We denote by $\mathcal{C}$ the open semi-infinite cylinder with base $\Omega$, by $\partial_L\mathcal{C}$ the lateral boundary of $\mathcal{C}$, and by $\mathcal{C}_\Omega$ the cylinder $\mathcal{C}$ with the base $\Omega$, that is,
\begin{equation*}
\mathcal{C}=\Omega\times(0,\infty),\qquad\qquad
\partial_L\mathcal{C}=\partial\Omega\times[0,\infty),\qquad\qquad
\mathcal{C}_\Omega=\mathcal{C}\cup(\Omega\times\{0\}).
\end{equation*}
A generic point $X$ in $\mathbb{R}^{N+1}$ is denoted by $(x,y)$, where $x\in\mathbb{R}^N$ and $y\in\mathbb{R}$.

A function $\rho$ is said to be a weight if $\rho$ is positive and finite almost everywhere. For an open set $U$, and a weight $\rho$, we denote by $L^p(U,\rho)$ the space of measurable functions $u:U\to \mathbb{R}$ such
\begin{equation*}
	\|u\|_{L^p(U,\rho)}:=\left(\int_U |u(x)|^p\rho(x)\dif x\right)^{1/p}<+\infty.
\end{equation*}
The space $L^p(U,\rho)$  endowed with the norm $\|\cdot\|_{L^p(U,\rho)}$ is a Banach space. Further, given $p\in [2,+\infty)$ we say that a weight $\rho$ satisfies the $B_p$ condition, and write $\rho\in B_p$, if $\rho^{-1/{p-1}}$ is locally integrable, that is, 
\begin{equation*}
	\rho\in B_p \qquad \Leftrightarrow \qquad \rho^{-1/(p-1)}\in L^1_{\mathrm{loc}}(U).
\end{equation*}
For a weight $\rho\in B_p$, we define the weighted Sobolev space $W^{1,p}(U,\rho)$ as the subset of $L^p(U,\rho)$ of functions $u$ with weak gradients $\nabla u$ such that  $|\nabla u|\in L^p(U,\rho)$.  Endowed with the norm 
\begin{equation*}
	\|u\|_{W^{1,p}(U,\rho)}:=\left(\int_U |u(x)|^p\rho(x)\dif x+\int_U |\nabla u(x)|^p\rho(x)\dif x\right)^{1/p}<+\infty,
\end{equation*}
$W^{1,p}(U,\rho)$ is a Banach space; see \cite{Ku1984}. Notice that $B_p$ is a larger class of weights than the Muckenhoupt weights $A_p$. The latter is also used to define weighted Sobolev spaces; see \cite{turesson2000nonlinear}. Throughout the paper we assume $p\in[2,\infty)$ and denote the  (H\"older) conjugate exponent of $p$ by $p'$. 

The measurable function $s(\cdot):\Omega\to\mathbb{R}$, which will characterize the spatially variable order of the fractional Laplacian, is assumed to satisfy:   
\begin{enumerate}[label=(\sc{H\arabic*})]
\item[]
\item \label{H1} $s(x)\in[0,1]$ for almost all $x\in\Omega$.
\item[]
\end{enumerate}
We use the notation $s(\cdot)$ to emphasize the dependence of the function $s:\Omega\to\mathbb{R}$ on the spatial variable $x \in \Omega$, and use $s$ to denote a constant in the interval $(0,1)$.        

Throughout the paper we consider the function $w:\mathcal{C}\to\mathbb{R}$ defined by
\begin{equation*}
w(x,y)=\mathrm{G}_s(x)y^{1-2s(x)},
\end{equation*}
and such that for a given $s(\cdot)$, and $p$, the function $\mathrm{G}_s:\Omega\to\mathbb{R}$ satisfies that 
\begin{enumerate}[label=(\sc{H\arabic*}),resume]
\item[]
\item \label{H2} $\mathrm{G}_s\in B_p$, and if $s(\cdot)= s\in(0,1)$ constant, then 
\begin{equation*}
	\mathrm{G}_s(x)= \frac{2^{2s-1}\Gamma(s)}{\Gamma(1-s)},
\end{equation*} for all $x\in\Omega$. Here $\Gamma$ is the standard Euler-Gamma function.
\item[]
\end{enumerate}
Assumptions \ref{H1} and \ref{H2} imply that  $w\in B_p$. However, it is known that (in general) $w$ is not expected to be of Muckenhoupt type, see \cite[Proposition~1]{AnRa2019}.

Given $\tau>0$, we denote by $\mathcal{C}^\tau$ the truncated cylinder $\mathcal{C}$ of height $\tau$, that is,
\begin{equation*}
\mathcal{C}^\tau=\Omega\times(0,\tau),
\end{equation*}
and define the sets $\partial_L\mathcal{C}^\tau$ and $\mathcal{C}^\tau_\Omega$ accordingly. The restriction of the weight $w$ to $\mathcal{C}^\tau$ is also denoted by $w$.

\begin{example}\label{Ex:sG} 
A possible choice for the function $s(\cdot)$ is given by
$$s(x)=\sigma\min(\mathrm{dist}(x,\mathcal{B}),\varepsilon ),$$ 
where  $0<\varepsilon<1$,
$\mathcal{B}\subset \Omega$ is a closed subset with zero-measure of $\mathbb{R}^N$ and $\mathrm{dist}(x,\mathcal{B})=\inf\{|x-y|:y\in \mathcal{B}\}$ and $\sigma\in (0,1)$. 
This type of functions are useful in image processing where the set $\mathcal{B}$ is the approximated set of edges/discontinuities of a certain image that one tries 
to recover; see \cite{AnRa2019}.

The two examples for $\mathrm{G}_s$ that are of relevance to us are defined by
\begin{equation}\label{Eq:Gexamples}
\mathrm{G}^{(1)}_s(x)\,=\,2^{2\overline{s}-1}\frac{\Gamma(\overline{s})}{\Gamma(1-\overline{s})} \qquad \text{and}\qquad 
\mathrm{G}^{(2)}_s(x)\,=\,2^{2s(x)-1}\frac{\Gamma(s(x))}{\Gamma(1-s(x))},
\end{equation}
where $\overline{s}=\displaystyle\frac{1}{|\Omega|}\int_{\Omega} s(x)\mathrm{d}x$.  It follows that {\normalfont\ref{H2}} is satisfied given that $\sigma\in (0,1)$. 
\end{example}


\section{The extended domain approach}\label{Sect:Constant}
{This section is devoted to briefly review the well-known extension domain approach to define the spectral fractional Laplacian, see for instance \cite{CaSi2007,StTo2010,EDNezza_GPalatucci_EValdinoci_2012a}.}
Throughout this section, we assume that $s\in(0,1)$ is constant. 

We denote by $\{\lambda_n\}$ the sequence of eigenvalues of the Laplace operator supplemented with a Dirichlet boundary condition, and consider an orthonormal basis $\{\varphi_n\}$ of $L^2(\Omega)$ of associated eigenfunctions. The spectral fractional Laplacian is defined by
\begin{equation}\label{Eq:FractLaplacian}
(-\Delta)^sv=\displaystyle\sum_{n=1}^\infty
\lambda_n^sb_n\varphi_n\qquad\text{where}\quad b_n=\int_\Omega v\varphi_n\,\dd x,
\end{equation} 
on the space
\begin{equation*}
H=\left\{v=\displaystyle\sum_{n=1}^\infty b_n\varphi_n\in L^2(\Omega):\|v\|^2_H=\displaystyle\sum_{n=1}^\infty
\lambda_n^sb_n^2<\infty\right\}.
\end{equation*}
For extensions of \eqref{Eq:FractLaplacian} to non-homogeneous boundary conditions, we refer to \cite{AnPfRo2017}. 
{It is worth mentioning that {$H = \mathbb{H}^s_0(\Omega)$} if $s\in\left(0,\frac{1}{2}\right)\text{ or }s\in\left(\frac{1}{2},1\right)$
and $H = \mathbb{H}^s_{00}(\Omega)$ for $s=\frac{1}{2}$.} 
Here, $\mathbb{H}_0^s(\Omega)$ is the closure in $\mathbb{H}^s(\Omega)$ of the space of infinitely continuous differentiable functions with compact support in $\Omega$, and $\mathbb{H}^s_{00}(\Omega)$ is the Lions-Magenes space \cite{LTartar_2007a}.
{Moreover,} $\mathbb{H}^s(\Omega)$ is the fractional Sobolev space of order $s$, 
\begin{equation*}
\mathbb{H}^s(\Omega)=\left\{v\in L^2(\Omega):\int_\Omega\int_\Omega\frac{|v(x)-v(y)|^2}{|x-y|^{N+2s}}\,\dd x\,\dd y<\infty\right\},
\end{equation*}
endowed with the norm
\begin{equation*}
\|v\|_{\mathbb{H}^s(\Omega)}=\left(\int_\Omega|v|^2\,\dd x+\int_\Omega\int_\Omega\frac{|v(x)-v(y)|^2}{|x-y|^{N+2s}}\,\dd x\,\dd y\right)^{1/2}.
\end{equation*}

The extension approach introduced by Caffarelli and Silvestre \cite{CaSaSi2007}, see \cite{StTo2010,CaEtAl2011} for the case of bounded domains, establishes that if $h\in H'$ (dual space of $H$) then the unique solution to the elliptic equation
\begin{equation*}
\begin{split}
(-\Delta)^sv&=h \quad \text{ in } \Omega,\\
 v& = 0 \quad  \text{ on } \partial \Omega,
\end{split}
\end{equation*}
is given by $v=\tr\,u$, where $u\in H^1_{0,L}(\mathcal{C},y^{1-2s})$ satisfies
\begin{equation}\label{P00}
{ \langle h,\tr\psi\rangle_{H',H}}=\frac{2^{2s-1}\Gamma(s)}{\Gamma(1-s)}\int_{\mathcal{C}}y^{1-2s}\nabla u\cdot\nabla \psi\,\dif X,\qquad\qquad\forall\,\psi\in H^1_{0,L}(\mathcal{C},y^{1-2s}),
\end{equation}
see \cite[Lemma 2.2]{CaEtAl2011}. Here, {$\langle \cdot, \cdot \rangle_{H',H}$} denotes the dual pairing between $H'$ and $H$. Moreover, $\tr$ is the $\Omega$-trace operator for functions in the space
\begin{equation*}
H^1_{0,L}(\mathcal{C},y^{1-2s})=\left\{u\in H^1(\mathcal{C},y^{1-2s}):u=0\text{ on }\partial_L\mathcal{C}\text{ in the trace sense}\right\}.
\end{equation*}
More precisely,  
\begin{equation*}
\tr:\,H^1_{0,L}(\mathcal{C},y^{1-2s})\to \mathbb{H}_0^s(\Omega),
\end{equation*}
is the unique bounded linear operator that satisfies $\tr u=u(\,\cdot\,,0)$ for every $u\in C^\infty(\bar{\mathcal{C}})$ that vanishes on $\partial_L\mathcal{C}$; which is also onto over $H$, that is 
\begin{equation}\label{Eq:Trace}
\tr H^1_{0,L}(\mathcal{C},w)=H,
\end{equation} 
see \cite[Proposition 2.1]{CaEtAl2011}.

Additionally, since the minimization problem
\begin{equation}\label{P0}
\begin{split}
&\mathrm{minimize}\quad \frac{1}{2}\displaystyle\int_{\mathcal{C}}y^{1-2s}\,|\nabla u|^2\,\dd X\quad \mathrm{over}\quad  H^1_{0,L}(\mathcal{C},y^{1-2s}),\\
&\mathrm{subject\:\:to}\quad \tr u=v,
\end{split}
\end{equation}
admits a unique solution $u\in H^1_{0,L}(\mathcal{C},y^{1-2s})$ for any $v\in\tr H^1_{0,L}(\mathcal{C},w)$, the harmonic extension operator
\begin{equation*}
\mathcal{S}: \tr H^1_{0,L}(\mathcal{C},y^{1-2s}) \to  H^1_{0,L}(\mathcal{C}, y^{1-2s}),\qquad
v\mapsto \mathcal{S}(v)=u,
\end{equation*}
where $u$ is the solution to problem \eqref{P0}, is well-defined, linear, and bounded. Then one finds that the spectral fractional Laplacian given by \eqref{Eq:FractLaplacian} satisfies
\begin{equation}\label{Eq:StartingPoint}
{\langle (-\Delta)^sv,\tr\psi\rangle_{H',H}}=\frac{2^{2s-1}\Gamma(s)}{\Gamma(1-s)}\int_{\mathcal{C}}y^{1-2s}\nabla \mathcal{S}(v)\cdot\nabla \psi\,\dif X,
\end{equation}
for all $\psi\in H^1_{0,L}(\mathcal{C},y^{1-2s})$ and all $v\in H$, which provides an equivalent definition for $(-\Delta)^s$. This second approach is our starting point to study the fractional Laplacian with spatially variable order: We identify a space of traces on which we can define the fractional Laplacian $(-\Delta)^{s(\cdot)}$ by a formula analogous to \eqref{Eq:StartingPoint}. 


\section{Abstract definition and solution to $(-\Delta)^{s(\cdot)}v=h$}\label{Sect:Abstract}
We consider in this section an abstract derivation of the spatially variable fractional Laplacian $(-\Delta)^{s(\cdot)}$. The advantage of this initial approach is that it requires minimal assumptions, namely \ref{H1} and \ref{H2}, which are primarily sufficient conditions to have $w\in B_p$; this leads to an appropriate definition of the associated weighted Sobolev spaces. Also, it is worth noticing that the arguments in this section do not require any assumption on the regularity of the $\Omega$ boundary {$\partial\Omega$}. 
This path starts with the proper derivation of the trace space for the weighted Sobolev spaces in study. For this matter, 
{we consider the space}
\begin{equation*}
L^{1,2}(\mathcal{C},w)=\{u:\mathcal{C}\to\mathbb{R} \text{ measurable : } \nabla u\in L^2(\mathcal{C},w)  \},
\end{equation*}
and endow it with the semi-norm
\begin{equation*}
	\|u\|_{L^{1,2}(\mathcal{C},w)}:=\|\nabla u\|_{L^2(\mathcal{C},w)}.
\end{equation*}
Note that $u\mapsto \|u\|_{L^{1,2}(\mathcal{C},w)}$ is a norm on the subset of $C^1$ functions in $L^{1,2}(\mathcal{C},w)$ that vanish at $\partial \mathcal{C}$ or $\partial_L \mathcal{C}$. Subsequently, we define  $\mathscr{L}_{0,L}^{1,2}(\mathcal{C},w)$ and $\mathscr{L}_0^{1,2}(\mathcal{C},w)$ as the completion in $L^{1,2}(\mathcal{C},w)$ of the infinitely differentiable functions in $L^{1,2}(\mathcal{C},w)$ with compact support in $\mathcal{C}_\Omega$ and $\mathcal{C}$, respectively, that is:
\begin{align*}
\mathscr{L}^{1,2}_{0,L}(\mathcal{C},w):=\,& \:\text{completion of }\:{C^\infty_c(\mathcal{C}_\Omega)\cap L^{1,2}(\mathcal{C},w)} \:\text{ for }\: \|\cdot \|_{L^{1,2}(\mathcal{C},w)},\\[0.2cm]
\mathscr{L}^{1,2}_0(\mathcal{C},w):=\,& \:\text{completion of }\:{C^\infty_c(\mathcal{C})\cap L^{1,2}(\mathcal{C},w)} \:\text{ for   }\: \|\cdot \|_{L^{1,2}(\mathcal{C},w)},
\end{align*}
where
\begin{equation}\label{eg:Cinf}
C^\infty_c(\mathcal{C}_\Omega)=\{u\in C^\infty(\bar{\mathcal{C}}): \mathrm{supp}(u)\cap  \partial_L\mathcal{C}=\emptyset \}.
\end{equation}
The only portion of the boundary where functions in  $C^\infty_c(\mathcal{C}_\Omega)$ do not necessarily vanish is the $\Omega$ cap.  
A few words are in order concerning $\mathscr{L}_{0,L}^{1,2}(\mathcal{C},w)$ and $\mathscr{L}_0^{1,2}(\mathcal{C},w)$. Note that $C^\infty_c(\mathcal{C}_\Omega)\cap L^{1,2}(\mathcal{C},w)$ and $C^\infty_c(\mathcal{C})\cap L^{1,2}(\mathcal{C},w)$ are both pre-Hilbert spaces when endowed with the inner product
\begin{equation*}
	(u_1,u_2)_{_{L^{1,2}(\mathcal{C},w)}}=\int_{\mathcal{C}}w\:\nabla u_1\cdot\nabla u_2\dif X.
\end{equation*}
It follows then that their completion, $\mathscr{L}_{0,L}^{1,2}(\mathcal{C},w)$ and $\mathscr{L}_0^{1,2}(\mathcal{C},w)$, are Hilbert spaces; in particular for {$z_1,z_2\in \mathscr{L}_{0,L}^{1,2}(\mathcal{C},w)$ there} exist Cauchy sequences $\{z_1^n\}$ and $\{z_2^n\}$  in $C^\infty_c(\mathcal{C}_\Omega)\cap L^{1,2}(\mathcal{C},w)$ such that
\begin{equation*}
	(z_1,z_2)_{\mathscr{L}_{0,L}^{1,2}(\mathcal{C},w)}:=\lim_{n\to \infty}\int_{\mathcal{C}}w \:\nabla z_1^n\cdot\nabla z^n_2\dif X.
\end{equation*}
If there is no risk of confusion, and in order to simplify notation, occasionally we simply write 
\begin{equation*}
	(z_1,z_2)_{\mathscr{L}_{0,L}^{1,2}(\mathcal{C},w)}= \int_{\mathcal{C}}w \:\nabla z_1\cdot\nabla z_2\dif X, 
\end{equation*}
and analogously we treat $\mathscr{L}_{0}^{1,2}(\mathcal{C},w)$.

Given that $C^\infty_c(\mathcal{C})\cap L^{1,2}(\mathcal{C},w)\subset C^\infty_c(\mathcal{C}_\Omega)\cap L^{1,2}(\mathcal{C},w)$, then we observe that  $\mathscr{L}_0^{1,2}(\mathcal{C},w)$ is a closed subspace of $\mathscr{L}_{0,L}^{1,2}(\mathcal{C},w)$. Thus, we can define an abstract space of traces on $\Omega$ of functions in $\mathscr{L}_{0,L}^{1,2}(\mathcal{C},w)$ as the quotient space 
\begin{equation*}
\mathscr{X}(\Omega,w):=\mathscr{L}_{0,L}^{1,2}(\mathcal{C},w)/\mathscr{L}_0^{1,2}(\mathcal{C},w).
\end{equation*}
We then define 
\begin{equation*}
	\Tr u:=[u],
\end{equation*} i.e., the abstract trace on $\Omega$ of a function $u\in \mathscr{L}_{0,L}^{1,2}(\mathcal{C},w)$ is identified with the equivalence class $[u]$ that contains $u$. The space $\mathscr{X}(\Omega,w)$ is then endowed  with the usual norm
\begin{align*}
\|\Tr u\|_{\mathscr{X}(\Omega,w)}=\|[u]\|_{\mathscr{X}(\Omega,w)}:=\,&\inf\{ \|u-z\|_{\mathscr{L}_{0,L}^{1,2}(\mathcal{C},w))}:z\in \mathscr{L}_0^{1,2}(\mathcal{C},w)\}.
\end{align*}

Note that
\begin{equation}
	\Tr :\mathscr{L}_{0,L}^{1,2}(\mathcal{C},w)\to \mathscr{X}(\Omega,w),
\end{equation}
is a linear and bounded operator, and that $\mathscr{X}(\Omega,w)$ is a Hilbert space, given that $\mathscr{L}_{0,L}^{1,2}(\mathcal{C},w)$ and $\mathscr{L}_0^{1,2}(\mathcal{C},w)$ are also Hilbert spaces. We denote its inner product as $(\cdot,\cdot)_\mathscr{X}$.  Further notice that, by definition, $\Tr \mathscr{L}_{0,L}^{1,2}(\mathcal{C},w)=\mathscr{X}(\Omega,w)$. Unless it is not clear from the context, we denote the class $[v]\in\mathscr{X}(\Omega,w)$ simply by $v$. The following result establishes the existence of the harmonic extension operator.

\begin{theorem}\label{Thm:AbstractMinim}
Let $v\in \mathscr{X}(\Omega,w)$ and $\mu>0$. The minimization problem: 
\begin{equation}\label{Pmu}\tag{\text{$ \mathbb{P} _{\mu,v}$}}
\mathrm{minimize}\quad  J_\mu(u,v)\quad\mathrm{ over }\quad  \mathscr{L}_{0,L}^{1,2}(\mathcal{C},w),
\end{equation}
for
\begin{equation*}
J_\mu(u,v):=\frac{1}{2}\|u\|^2_{\mathscr{L}_{0,L}^{1,2}(\mathcal{C},w)}+
\frac{\mu}{2}\|\Tr u-v\|_{\mathscr{X}(\Omega,w)}^2,
\end{equation*}
admits a unique solution $u_\mu\in \mathscr{L}_{0,L}^{1,2}(\mathcal{C},w)$ that, as $\mu\to \infty$, converges strongly to the unique solution to
\begin{equation}\label{P}\tag{\text{$\mathbb{P}_v$}}
\begin{split}
&\mathrm{minimize}\quad J(u)\quad \mathrm{over}\quad  \mathscr{L}_{0,L}^{1,2}(\mathcal{C},w), \\
&\mathrm{subject\:\:to}\quad \Tr u=v,
\end{split}
\end{equation}
for
\begin{equation*}
J(u):=\frac{1}{2}\|u\|^2_{\mathscr{L}_{0,L}^{1,2}(\mathcal{C},w)}.
\end{equation*}
\end{theorem}

\begin{proof}
The existence of a solution $\{u_\mu\}$ to \eqref{Pmu} follows from arguments of the direct methods for calculus of variations: The functional $u\mapsto J_\mu(u,v)$ is non-negative, coercive, and weakly lower semicontinuous; for the latter part note that $\mathscr{L}_{0,L}^{1,2}(\mathcal{C},w)\ni w\mapsto \|\Tr w\|_{\mathscr{X}(\Omega,w)}$ is also weakly lower semicontinuous. Uniqueness follows from the strict convexity of $u\mapsto J_\mu(u,v)$. 

Since $v\in \mathscr{X}(\Omega,w)$, there exists $\tilde{u}\in  \mathscr{L}_{0,L}^{1,2}(\mathcal{C},w)$ such that $v=[\tilde{u}]=\Tr \tilde{u}$. Thus, given that $u_\mu$ is a minimizer of $J_\mu(\cdot, v)$, 
\begin{equation}\label{eq:umubound}
	J_\mu(u_\mu,v)\leq J_\mu(\tilde{u},v)=J(\tilde{u}),
\end{equation}
for every $\mu>0$. Then, by basic theory for penalty functions (see \cite[Lemma 1 in Chapter 10]{Luenberger:1997:OVS:524037}) we have that  
\begin{equation}\label{eq:P00}
\lim_{\mu\to\infty} \frac{\mu}{2}\|\Tr u_\mu-v\|_{\mathscr{X}(\Omega,w)}^2=0.
\end{equation} 

Thus, by \eqref{eq:umubound} we have that the sequence $\{u_\mu\}$ is bounded in $\mathscr{L}_{0,L}^{1,2}(\mathcal{C},w)$, so it admits a weakly convergent subsequence, say 
\begin{equation}\label{App-Eq:Weak}
u_{\mu'}\rightharpoonup u\qquad\text{in}\quad\mathscr{L}_{0,L}^{1,2}(\mathcal{C},w).  
\end{equation}
Further, by \eqref{eq:P00} we observe that $\Tr u=v$. Next we show that $J(u_\mu) \rightarrow J(u)$ with $u$ being the minimizer to \eqref{P}. 
By weak lower semicontinuity of $J$ and \eqref{eq:P00}, we observe:
\begin{align*}
J(u)&\leq \varliminf_{\mu'\to\infty} J(u_{\mu'})\leq \varlimsup_{\mu'\to\infty} J(u_{\mu'})=\varlimsup_{\mu'\to\infty} J_{\mu'}(u_{\mu'},v)\leq \varlimsup_{\mu'\to\infty} J_{\mu'}(u,v)=J(u),
\end{align*}
that is $J(u_{\mu'})\to J(u)$. The fact that $u$ is a minimizer to \eqref{P} follows by selecting an arbitrary $\tilde{u}$ such that $\Tr \tilde{u} =v$, then the previous to last inequality above yield 
\begin{align*}
J(u)& \leq \varlimsup_{\mu'\to\infty} J_{\mu'}(\tilde{u},v)=J(\tilde{u}),
\end{align*}
i.e.,  $u$ is a minimizer. Further, by strict convexity, minimizers to \eqref{P} are unique, so that the entire sequence $\{u_{\mu}\}$ satisfies 
\begin{equation}\label{App-Eq:Weak2}
u_{\mu}\rightharpoonup u\qquad\text{in}\quad\mathscr{L}_{0,L}^{1,2}(\mathcal{C},w),
\end{equation}
and also $J(u_\mu)\to J(u)$. Using \eqref{eq:P00}, this limit is equivalent to
\begin{equation*}
\lim_{\mu\to\infty}\|u_\mu\|_{\mathscr{L}_{0,L}^{1,2}(\mathcal{C},w)}=\|u\|_{\mathscr{L}_{0,L}^{1,2}(\mathcal{C},w)},
\end{equation*}
which together with \eqref{App-Eq:Weak2} implies that
\begin{equation}\label{App-Eq:Strong}
u_{\mu}\to u\qquad\text{in}\quad\mathscr{L}_{0,L}^{1,2}(\mathcal{C},w),
\end{equation}
 see \cite[Proposition 3.32]{Br2011}. 
\end{proof}

Theorem \ref{Thm:AbstractMinim} ensures the existence of the abstract weighted harmonic extension operator 
\begin{equation*}
S: \Tr \mathscr{L}_{0,L}^{1,2}(\mathcal{C},w) \to  \mathscr{L}_{0,L}^{1,2}(\mathcal{C},w),\qquad
v\mapsto S(v)=u.
\end{equation*}
where $u$ is the solution to \eqref{P}.
In addition, the map $S$ is linear and bounded: Linearity follows directly from the examination of the first order conditions. For boundedness, consider \eqref{eq:umubound} with $u-z$ instead of $\tilde{u}$, where $u$ solves \eqref{P} and $z\in\mathscr{L}_{0}^{1,2}(\mathcal{C},w)$, to obtain $J_\mu(u_\mu,v)\leq J(u-z)$. Then, by taking the limit as $\mu\to \infty$ we observe 
\begin{equation*}
	 \|S(v)\| _{\mathscr{L}_{0,L}^{1,2}(\mathcal{C},w)}=\|u\| _{\mathscr{L}_{0,L}^{1,2}(\mathcal{C},w)}\leq \|u-z\| _{\mathscr{L}_{0,L}^{1,2}(\mathcal{C},w)}.
\end{equation*}
Then, by considering the infimum over all $z\in \mathscr{L}_{0}^{1,2}(\mathcal{C},w)$, we obtain
\begin{equation*}
	\|S(v)\| _{\mathscr{L}_{0,L}^{1,2}(\mathcal{C},w)}\leq \|v\|_{\mathscr{X}(\Omega,w)}.
\end{equation*}

The well-posedness of the map $S$ allows us to establish a definition for the fractional Laplacian with spatially variable order.
   
\begin{definition}\label{Def:AbstractLaplacian}
Let $\mathscr{X}(\Omega,w)'$ be the dual space of $\mathscr{X}(\Omega,w)$. The operator 
\begin{equation*}
	(-\Delta)^{s(\cdot)}:\mathscr{X}(\Omega,w)\to\mathscr{X}(\Omega,w)',
\end{equation*} is determined as follows: for $v \in \mathscr{X}(\Omega,w)$, then $(-\Delta)^{s(\cdot)}v\in \mathscr{X}(\Omega,w)'$ is defined by 
\begin{equation}\label{Eq:AbstractLaplacian}
\langle(-\Delta)^{s(\cdot)}v, \Tr \psi \rangle_{\mathscr{X}',\mathscr{X}}=( S(v), \psi)_{\mathscr{L}_{0,L}^{1,2}(\mathcal{C},w)},\qquad\qquad\forall\, \psi\in \mathscr{L}_{0,L}^{1,2}(\mathcal{C},w).
\end{equation}
\end{definition}

\begin{remark}
\label{rem:DtN}
The relation of the above definition with the classical spectral fractional Laplacian \eqref{Eq:StartingPoint} is straightforward in light of the abuse of notation disclosed at the beginning of the chapter; in which case we can write
\begin{equation*}
\langle(-\Delta)^{s(\cdot)}v, \Tr \psi \rangle_{\mathscr{X}',\mathscr{X}}=\displaystyle\int_{\mathcal{C}}w\,\nabla S(v)\cdot\nabla \psi\,\dd X,\qquad\qquad\forall\,\psi\in \mathscr{L}_{0,L}^{1,2}(\mathcal{C},w).
\end{equation*}
Furthermore, by a formal integration-by-parts formula and using the fact that $S(v)$ is weighted harmonic, 
we obtain that $(-\Delta)^{s(\cdot)}$ is equal to the generalized Neumann trace of $S(v)$ when restricted
to $\Omega \times \{0\}$. 
{Then, as for the cassical case with constant order $s$, $(-\Delta)^{s(\cdot)}$ can be understood as a Dirichlet-to-Neumann map.}
\end{remark}

\begin{remark}
	In view of Theorem \ref{Thm:AbstractMinim}, the expression in \eqref{Eq:AbstractLaplacian} is equivalent to
	\begin{equation*}\label{Eq:AbstractLaplacianMu}
\langle(-\Delta)^{s(\cdot)}v, \Tr \psi \rangle_{\mathscr{X}',\mathscr{X}}=\lim_{\mu\to\infty}\mu( \Tr u_\mu-v, \Tr \psi)_{\mathscr{X}},\qquad\forall\, \psi\in \mathscr{L}_{0,L}^{1,2}(\mathcal{C},w),
\end{equation*}
where $u_\mu$ is the unique solution to \eqref{Pmu}.
\end{remark}

The operator $(-\Delta)^{s(\cdot)}:\mathscr{X}(\Omega,w)\to\mathscr{X}(\Omega,w)'$ is well-defined as we see next, and it can be seen as the Lagrange multiplier associated to the harmonic extension problem.
\begin{proposition}\label{prop:Lagrange}
For each $v \in \mathscr{X}(\Omega,w)$, there exists a unique $\lambda=\lambda(v)\in  \mathscr{X}(\Omega,w)'$ such that 
\begin{equation*}
\langle\lambda , \Tr \psi \rangle_{\mathscr{X}',\mathscr{X}}=(S(v),\psi)_{\mathscr{L}_{0,L}^{1,2}(\mathcal{C},w)},\qquad\qquad\forall\,\psi\in \mathscr{L}_{0,L}^{1,2}(\mathcal{C},w).
\end{equation*}
\end{proposition}
\begin{proof}
Initially, note that $S(v)$ is the solution to \eqref{P}. For convenience, we write the constraint in \eqref{P} as $G(u)=0$, where $G: \mathscr{L}_{0,L}^{1,2}(\mathcal{C},w)\to \mathscr{X}(\Omega,w)$ is defined by $G(u)=\Tr u-v$. Since the operator $\Tr$ is linear and bounded, also it is $G$, and hence $G'(u)h=\Tr h$. Thus, $G'(u):\mathscr{L}_{0,L}^{1,2}(\mathcal{C},w)\to \mathscr{X}(\Omega,w)$ is linear, bounded, and surjective. Therefore, there exists a unique Lagrange multiplier $\lambda\in   \mathscr{X}(\Omega,w)'$ such that 
\begin{equation*}
	J'(S(v))\psi=\lambda \circ G'(S(v))\psi,\qquad\qquad\forall\,\psi\in \mathscr{L}_{0,L}^{1,2}(\mathcal{C},w),
\end{equation*}
which proves the statement. 
\end{proof}

In view of Remark \ref{rem:DtN}, we can also interpret $\lambda$ as the Neumann trace of the extension
onto $\Omega \times \{0\}$.

\begin{remark}
It follows that $(-\Delta)^{s(\cdot)}:\mathscr{X}(\Omega,w)\to\mathscr{X}(\Omega,w)'$  is a bounded linear operator given that $S$ is linear and bounded.	
\end{remark}

We are now able to determine existence of solutions to the Poisson problem with spatially variant Laplacian.

\begin{theorem}\label{Thm:AbstractElliptic} Let $h\in \mathscr{X}(\Omega,w)'$. The equation 
\begin{equation}\label{Eq:AbstractEq}
(-\Delta)^{s(\cdot)}v=h,
\end{equation}
admits a unique solution in $\mathscr{X}(\Omega,w)$ that is given by $v=\Tr\,u$, where $u$ solves
\begin{align}\label{Eq:AbstractMinEllip}
\mathrm{minimize}\quad \mathcal{J}(u)\quad \mathrm{over}\quad \mathscr{L}_{0,L}^{1,2}(\mathcal{C},w),
\end{align}
for 
\begin{equation*}
\mathcal{J}(u):=\frac{1}{2}\|u\|^2_{\mathscr{L}_{0,L}^{1,2}(\mathcal{C},w)} -\langle h, \Tr u\rangle_{\mathscr{X}', \mathscr{X}}.
\end{equation*} 
\end{theorem}

\begin{proof}
Since $\Tr$ is linear and bounded, we have that
	\begin{equation*}
		u\mapsto \langle h, \Tr u\rangle_{\mathscr{X}',\mathscr{X}},
	\end{equation*} 
is a linear functional over $\mathscr{L}_{0,L}^{1,2}(\mathcal{C},w)$. Then, there exists a solution  to the problem \eqref{Eq:AbstractMinEllip} and the solution is unique due to strict convexity of $\mathcal{J}$.
	
	Note that via necessary and sufficient conditions of optimality for \eqref{Eq:AbstractMinEllip}, the unique solution $u$ satisfies:
\begin{equation}
(u,\psi)_{\mathscr{L}_{0,L}^{1,2}(\mathcal{C},w)}=\langle h, \Tr \psi \rangle_{\mathscr{X}',\mathscr{X}},\qquad\forall\,\psi \in \mathscr{L}_{0,L}^{1,2}(\mathcal{C},w),
\label{Proof:AbstractOptCond}
\end{equation}
and then $u$ is identical to its harmonic extension, i.e., $u=S(\Tr u)$. To see the latter, we consider $\psi\in C_c^\infty(\mathcal{C})\cap \mathscr{L}_{0,L}^{1,2}(\mathcal{C},w)$ in \eqref{Proof:AbstractOptCond} and observe that by density
\begin{equation}\label{Eq:AbstractOptCondConseq}
(u,\psi)_{\mathscr{L}_{0,L}^{1,2}(\mathcal{C},w)}=0,\qquad\forall\,\psi \in \mathscr{L}_{0}^{1,2}(\mathcal{C},w),
\end{equation}
where we have used the fact that the functions in $C_c^\infty(\mathcal{C})$ vanish on $\Omega \times \{0\}$.
Moreover, we also (trivially) have $\Tr S(\Tr u)=\Tr u$ so that $u$ satisfies first order optimality conditions for \eqref{P} for $v=\Tr u$. 
Hence, by convexity (uniqueness) $u=S(\Tr u)$. Also, by definition of the operator $(-\Delta)^{s(\cdot)}$ and 
\eqref{Proof:AbstractOptCond}, we have 
\begin{equation}
\langle(-\Delta)^{s(\cdot)}\Tr u, \Tr \psi \rangle_{\mathscr{X}',\mathscr{X}}=(S(\Tr u), \psi)_{\mathscr{L}_{0}^{1,2}(\mathcal{C},w)}=\langle h, \Tr \psi \rangle_{\mathscr{X}',\mathscr{X}},
\end{equation}
for all $\psi\in \mathscr{L}_{0,L}^{1,2}(\mathcal{C},w)$ and hence $\Tr u$ solves \eqref{Eq:AbstractEq}.

To prove uniqueness, consider a solution $v$ to \eqref{Eq:AbstractEq} with $h=0$ and notice that 
\begin{equation*}
(S(v), \psi)_{\mathscr{L}_{0}^{1,2}(\mathcal{C},w)}=0,\qquad\forall\,\psi \in \mathscr{L}_{0,L}^{1,2}(\mathcal{C},w).
\end{equation*}
Then, $S(v)$ satisfies first order optimality conditions for 
\begin{equation*}
\mathrm{minimize}\quad  \frac{1}{2}\| u\|^2_{\mathscr{L}_{0}^{1,2}(\mathcal{C},w)} \quad\mathrm{ over }\quad  \mathscr{L}_{0,L}^{1,2}(\mathcal{C},w),
\end{equation*}
whose unique minimizer is the zero function. Then, by convexity, $S(v)=0$, so that $v=\Tr S(v)$ and hence $v=0$. 
\end{proof}

\begin{remark}[Truncated cylinder $\mathcal{C}^\tau$]
It is worth mentioning that exactly the same construction with $\mathcal{C}$ replaced by the truncated cylinder $\mathcal{C}^\tau$, $\tau>0$, leads to a definition of $(-\Delta)^{s(\cdot)}$ by means of an extension problem on $\mathcal{C}^\tau$, as well as to the existence and uniqueness of solution to the associated Poisson problem. We care about $\mathcal{C}^\tau$ because it makes the problem tractable from an implementation point of view \cite{AnRa2019,RHNochetto_EOtarola_AJSalgado_2014a}. 
\end{remark}

A few words are in order concerning Theorem \ref{Thm:AbstractElliptic}; although it provides a solvability result for the elliptic problem, it does not establish existence of solutions based on maps defined on $\Omega$. That is, we would like to address the question: Under what conditions on $h:\Omega\to\mathbb{R}$, does the equation $(-\Delta)^{s(\cdot)}v=h$ admit a solution? This question is answered in Section \ref{Sect:Elliptic} and it is intimately related to the following trace results.


\section{Trace theorems}\label{Sect:Traces}
In this section we identify a trace operator that properly relates values of maps on a Sobolev space in $\mathcal{C}$ to their values at $\Omega$. For this matter, in addition to \ref{H1} and \ref{H2}, we assume that the measurable function $s(\cdot)$ satisfies:
\begin{enumerate}[label=(\sc{H\arabic*}), resume]
\item[]
\item \label{H3} The set of points on which $s(\cdot)$ is zero has measure zero, i.e., $|A_0|=0$ where
\begin{equation*}
	A_0:=\{x\in \Omega : s(x)=0\}.
\end{equation*}
\item[]
\end{enumerate}

We define $\mathscr{H}_{0,L}^{1,p}(\mathcal{C},w)$ to be the closure in $W^{1,p}(\mathcal{C},w)$ of the infinitely differentiable functions in $W^{1,p}(\mathcal{C},w)$  
with compact support in $\mathcal{C}_\Omega$, that is,
\begin{equation*}\label{eq:HLinitial}
\mathscr{H}_{0,L}^{1,p}(\mathcal{C},w)=\overline{C^\infty_c(\mathcal{C}_\Omega)\cap W^{1,p}(\mathcal{C},w)}^{W^{1,p}(\mathcal{C},w)},
\end{equation*}
where $C^\infty_c(\mathcal{C}_\Omega)$ is given in \eqref{eg:Cinf}. Then, formally speaking, $\mathscr{H}_{0,L}^{1,p}(\mathcal{C},w)$ is the set of functions in $W^{1,p}(\mathcal{C},w)$ that vanish on $\partial_L\mathcal{C}$. We now prove the regularity of restrictions of functions in $\mathscr{H}_{0,L}^{1,p}(\mathcal{C},w)$ on the $\Omega$ boundary.

\begin{theorem}[\textsc{Trace theorem}]\label{Thm:FirstTrace}
Provided that {\normalfont\ref{H1}}, {\normalfont\ref{H2}}, and {\normalfont\ref{H3}} hold true, there exists a unique bounded linear operator
\begin{equation*}
\tr\,:\,\mathscr{H}_{0,L}^{1,p}(\mathcal{C},w)\to L^p(\Omega,\tilde{w}),
\end{equation*}
that satisfies $\tr u=u(\,\cdot\,,0)$ for all $u\in \mathscr{H}_{0,L}^{1,p}(\mathcal{C},w)\cap C^\infty_c(\mathcal{C}_\Omega)$, where the weight $\tilde{w}:\Omega\to\mathbb{R}$ is defined by
\begin{equation*}
\tilde{w}(x)=\mathrm{G}_s(x)(p-2+2s(x))^p.
\end{equation*}
 The same statement is true if we replace $\mathscr{H}_{0,L}^{1,p}(\mathcal{C},w)$ by the space $\mathscr{H}_{0,L}^{1,p}(\mathcal{C}^\tau,w)$,  for every $\tau>0$.
\end{theorem}

\begin{proof}
For the sake of brevity, we define $\delta(\cdot):=1-2s(\cdot)$ so that 
\begin{equation*}
\tilde{w}(x)= \mathrm{G}_s(x)(p-1-\delta(x))^p.
\end{equation*}
Let $u\in \mathscr{H}_{0,L}^{1,p}(\mathcal{C},w)\cap C_c^\infty(\mathcal{C}_\Omega)$ and $(x,y)\in \bar{\mathcal{C}}$ be such that $s(x)\neq 0$ and $\mathrm{G}_s(x)\neq 0$. Initially, we write
\begin{equation}\label{Proof:FirstTrace-1}
u(x,0)=u(x,y)-\int_0^1  y D_{N+1} u(x,ty) \dif t,
\end{equation}
where $D_{N+1}u$ is the partial derivative of $u$ with respect to the $(N+1)$ coordinate.

Let $\sigma\in(0,1)$. Multiplying \eqref{Proof:FirstTrace-1} by $w(x,y)^{1/p}$ and then integrating from $0$ to $\sigma$ with respect to $y$, we find:
\begin{equation*}
\begin{split}
|u(x,0)|\int _0^\sigma w(x,y)^{1/p} \dd y\leq\,&I_1+I_2,
\end{split}
\end{equation*}
where:
\begin{align*}
I_1:=\,&\int _0^\sigma  |u(x,y)|w(x,y)^{1/p} \dif y,\\
I_2:=\,&\int_0^1\int _0^\sigma   y |D_{N+1} u(x,ty)| w(x,y)^{1/p} \dif y\dif t.
\end{align*}

Now, we notice that $\int _0^\sigma w(x,y)^{1/p} \dd y=\mathrm{G}_s(x)\int _0^\sigma y^{\delta(x)/p} \dd y$ and that
\begin{equation*}
\begin{split}
\int _0^\sigma y^{\delta(x)/p} \dd y\geq
\int_0^\sigma y^{1/p}\,\dd y=
\frac{p}{1+p}\sigma^{\frac{1+p}{p}}>0,
\end{split}
\end{equation*}
since $\sigma\in(0,1)$ and $\delta(x)\leq 1$. Thus,
\begin{equation*}
|u(x,0)|\mathrm{G}_s(x)^{1/p}\leq\frac{p+1}{p\,\sigma^{(1+p)/p}}(I_1+I_2).
\end{equation*}
Multiplying the last expression by $(p-1-\delta(x))$, we obtain:
\begin{equation}\label{Proof:FirstTrace-2}
|u(x,0)|\tilde{w}(x)^{1/p}\leq\frac{p+1}{p\,\sigma^{(1+p)/p}}(p-1-\delta(x))(I_1+I_2).
\end{equation}

Next, we shall estimate $I_1$ and $I_2$. A direct use of the H\"older's inequality yields:
\begin{align*}
I_1\leq \sigma^{1/p'}\left(\int _0^\sigma  |u(x,y)|^pw(x,y) \dif y\right)^{1/p}.
\end{align*}
We now estimate $I_2$ in several steps. With the change of variables $y=zt^{-1}$ in the inner integral of $I_2$, we obtain:
\begin{align}\label{Proof:FirstTrace-3}
I_2
\leq\,&\sigma\int_0^1\int _0^{t\sigma} |D_{N+1} u(x,z)| w(x,z)^{1/p} t^{-1-\delta(x)/p}\dif z\dif t.
\end{align}
By adding and substracting $(1+\delta(x))/pp'$ in the exponent of $t$, we rewrite the r.h.s. of \eqref{Proof:FirstTrace-3} as
\begin{align*}
\sigma\int_0^1t^{-\frac{1+\delta(x)}{pp'}}\int _0^{t\sigma} F(x,z)\, t^{\frac{1-pp'+(1-p')\delta(x)}{pp'}}\dif z\dif t,
\end{align*}
where $F(x,z)=|D_{N+1} u(x,z) |w(x,z)^{1/p}$. Then, by the H\"older's inequality, we find:
\begin{equation*}
\begin{split}
I_2\leq\,&
\sigma\left(\int_0^1t^{-\frac{1+\delta(x)}{p}}\,\dd t\right)^{1/p'}\left(\int_0^1\left(\int_0^{\sigma t}F(x,z)\,t^{\frac{1-pp'+(1-p')\delta(x)}{pp'}}\,\dd z\right)^p\,\dd t\right)^{1/p}\\
=\,&\sigma\left(\frac{p}{p-1-\delta(x)}\right)^{1/p'}\left(\int_0^1\left(\int_0^{\sigma t}F(x,z)\,t^{\frac{1-pp'+(1-p')\delta(x)}{pp'}}\,\dd z\right)^p\,\dd t\right)^{1/p}.
\end{split}
\end{equation*}
Applying the H\"older's inequality on the integral with respect to $z$, we obtain:
\begin{equation*}
\begin{split}
I_2\leq&
\sigma \left(\frac{p}{p-1-\delta(x)}\right)^{1/p'}\left(\int_0^1(\sigma t)^{p/p'}\int_0^{\sigma t}F(x,z)^p\,t^{\frac{1-pp'+(1-p')\delta(x)}{p'}}\,\dd z\,\dd t\right)^{1/p}\\
\leq\,&\sigma \left(\frac{p}{p-1-\delta(x)}\right)^{1/p'}\left(\int_0^1(\sigma t)^{p/p'}\int_0^{\sigma}F(x,z)^p\,t^{\frac{1-pp'+(1-p')\delta(x)}{p'}}\,\dd z\,\dd t\right)^{1/p}\\
=\, &\sigma^{1+p/p'} \left(\frac{p}{p-1-\delta(x)}\right)^{1/p'}\left(\int_0^1t^{-\frac{1+\delta(x)}{p}}\,dt\right)^{1/p}\left(\int_0^{\sigma}F(x,z)^p\,\dd z\right)^{1/p}.
\end{split}
\end{equation*}
Finally, we have:
\begin{equation*}
I_2\leq\, \frac{p\,\sigma^{1+p/p'}}{p-1-\delta(x)}\left(\int_0^{\sigma}|D_{N+1}u(x,z)|^pw(x,z)\,\dd z\right)^{1/p}.
\end{equation*}

Using the above estimations for $I_1$ and $I_2$ in \eqref{Proof:FirstTrace-2}, and observing that $p-1-\delta(x)\leq p$, we obtain:
\begin{equation*}
\begin{split}
|u(x,0)|\tilde{w}(x)^{1/p}\leq\,&
(p+1)\sigma^{-2/p}\left(\int _0^\sigma  |u(x,y)|^pw(x,y) \dif y\right)^{1/p}\\
+&(p+1)\sigma^{p-1-1/p}\left(\int_0^{\sigma}|D_{N+1}u(x,z)|^pw(x,z)\,\dd z\right)^{1/p},
\end{split}
\end{equation*}
from which we have:
\begin{equation}\label{Proof:FirstTrace-4}
|u(x,0)|\tilde{w}(x)^{1/p}\leq\,
(p+1)\sigma^{-2/p'}\left(\int _0^\sigma  \left(|u(x,y)|^p+|D_{N+1}u(x,y)|^p\right)w(x,y)\,\dd y\right)^{1/p},
\end{equation}
since $\sigma^{p-1-1/p}\leq\sigma^{-2/p}$.

Raising inequality \eqref{Proof:FirstTrace-4} to the $p$ power and then integrating over $\Omega$, we find:
\begin{align*}
\int_{\Omega}|u(x,0)|^p\tilde{w}(x)\,\dif x
\leq\,&(p+1)^p\sigma^{-2p/p'}\left(\|u\|^p_{L^p(\mathcal{C},w)}+\|\nabla u\|^p_{L^p(\mathcal{C},w)}\right).
\end{align*}
Therefore, $u(\,\cdot\,,0)\in L^p(\Omega,\tilde{w})$ and
\begin{align*}
\|u(\,\cdot\,,0)\|_{L^p(\Omega,\tilde{w})}\leq C(p,\sigma)
\|u\|_{W^{1,p}(\mathcal{C},w)},
\end{align*}
where $C(p,\sigma)=(p+1)\sigma^{-2/p'}$. Notice that $\sigma$ is an arbitrary, but fixed, number in $(0,1)$, so that in this case we can fix $ C(p,\sigma)$ to depend only on $p$. The operator $\tr$ is the unique bounded linear extension of the mapping $u(x,y)\mapsto u(x,0)$ to $
\mathscr{H}_{0,L}^{1,p}(\mathcal{C},w)$.  

 Let us finally see that the same trace result holds true when we replace $\mathscr{H}_{0,L}^{1,p}(\mathcal{C},w)$ by $\mathscr{H}_{0,L}^{1,p}(\mathcal{C}^\tau,w^\tau)$, where $\tau>0$. If $\tau\geq 1$, it follows from \eqref{Proof:FirstTrace-4} that
\begin{equation*}
|u(x,0)|\tilde{w}(x)^{1/p}\leq\,
(p+1)\sigma^{-2/p'}\left(\int _0^\tau  \left(|u(x,y)|^p+|D_{N+1}u(x,y)|^p\right)w(x,y)\,\dd y\right)^{1/p},
\end{equation*}
from which, exactly as before, we find
\begin{align}\label{Proof:FirstTrace-5}
\|u(\,\cdot\,,0)\|_{L^p(\Omega,\tilde{w})}\leq C(p,\sigma)
\|u\|_{W^{1,p}(\mathcal{C}^\tau,w^\tau)}.
\end{align}
If, on the contrary, $0<\tau<1$, then we select $\sigma=\tau$ in \eqref{Proof:FirstTrace-4} and obtain \eqref{Proof:FirstTrace-5} in the same way. The trace operator is now obtained as before.
\end{proof}

\begin{remark}\label{RK:Classic1}
If $s(\cdot)=s\in(0,1)$ is constant, then both $\tilde{w}$ and $\mathrm{G}_s$ are also constants. Hence, $L^p(\Omega,\tilde{w})=L^p(\Omega)$ and  $\mathscr{H}^{1,p}_{0,L}(\mathcal{C},w)=\mathscr{H}^{1,p}_{0,L}(\mathcal{C},y^{1-2s})$, so it follows from Theorem \ref{Thm:FirstTrace} that
\begin{equation*}
\tr :\mathscr{H}^{1,p}_{0,L}(\mathcal{C},y^{1-2s})\to L^p(\Omega).
\end{equation*}
This is in accordance to the classical case,  {see \cite[Theorem 3.2]{Ne1993}}. 
If, additionally, $p=2$, then we observe that $\tr$ and the trace operator given in \cite{CaEtAl2011} (see also Section \ref{Sect:Constant}) coincide for functions in $C_c^\infty(\mathcal{C}_\Omega)$. From this, we find that $\tr$ is just given by the restriction to $\mathscr{H}_{0,L}^{1,2}(\mathcal{C},y^{1-2s})\subset H_{0,L}^1(\mathcal{C},y^{1-2s})$ of the map in \cite{CaEtAl2011}. However, a deeper result is true; see Theorem \ref{thm:surjectivitytrace}.
\end{remark}

In Theorem \ref{Thm:FirstTrace}, we have characterized the integrability of functions in the trace space of $\mathscr{H}_{0,L}^{1,p}(\mathcal{C},w)$. We aim now to  identify the ``smoothness'' of functions in this trace space. This is a more complicated task since we aim at determining a space with a spatially variable smoothness associated to the function $s(\cdot)$. 

For simplicity, from now on we assume that $\Omega$ is the $N$-dimensional unit square $\mathrm{Q}_N=(0,1)^N$. The forthcoming analysis requires one final assumption on the functions $s(\cdot)$ and $\mathrm{G}_s$:
\begin{enumerate}[label=(\sc{H\arabic*}),resume]
\item[]
\item\label{H5} For almost every $x_j,z\in(0,1),\,j\neq i,$ and $i=1,\hdots,N$, it holds true that 
\begin{equation*}
	\int_0^1\hspace*{-0.07cm}\left(\mathrm{G}_s(x)|x_i-z|^{1-2s(x)}\right)^{1-p'}\hspace*{-0.05cm}\dif x_i<\infty,
\end{equation*}
where $x=(x_1,\hdots,x_n)$.
\item[]
\end{enumerate}
Assumption \ref{H5} enables us to use a Hardy-type inequality (see Lemma \ref{Lem:HardyIneq} below) for two specially chosen weights, which is a key ingredient to prove the subsequent improvement of the trace result in Theorem \ref{Thm:SecondTrace}.

\begin{example}
Let $\Omega=\mathrm{Q}_1$, $p=2$, and $\mathrm{G}_s=\mathrm{G}_s^{(1)}$ constant; see \eqref{Eq:Gexamples} in Example \ref{Ex:sG}. Suppose that $s(\cdot)$ satisfies:
\begin{equation}\label{Eq:Example-s}
s(x)\geq m|x-x_0|^q\quad\text{if}\quad |x-x_0|\leq R,\qquad
s(x)>\mu>0\quad\text{if}\quad |x-x_0|>R,
\end{equation}
for some $q,R\in(0,1)$, $m,\mu>0$, and $x_0\in(R,1-R)$. Notice that the only point where $s$ is allowed to be zero is $x_0$. 
For this particular setting, although $w\notin A_p(\mathcal{C})$, i.e., $w$ is not a Muckenhoupt weight (see \cite{AnRa2019}), we find that {\em\ref{H5}} holds true as we see next. 

To simplify the notation below, we write $\delta(\cdot):=1-2s(\cdot)$. Since $\delta(x)(1-p')=-\delta(x)>-1$ for all $x\neq x_0$, we have:
\begin{equation*}
\begin{split}
\int_0^1 |x-z|^{-\delta(x)}\,\dd z=\,&\,\frac{x^{2s(x)}+(1-x)^{2s(x)}}{2s(x)},\qquad\forall\,x\neq x_0.
\end{split}
\end{equation*}
Then,
\begin{equation*}
\begin{split}
\int_0^1\left(\int_0^1|x-z|^{-\delta(x)}\,\dd z\right)\dd x
\leq\,&\int_0^1\frac{1}{s(x)}\,\dd x.
\end{split}
\end{equation*}
We now observe that, by \eqref{Eq:Example-s}, we have:
\begin{equation*}
\begin{split}
\int_0^1\frac{1}{s(x)}\,\dd x=\,&\int_0^{x_0-R}\frac{\dd x}{s(x)}+
\int_{x_0-R}^{x_0+R}\frac{\dd x}{s(x)}+
\int_{x_0+R}^1\frac{\dd x}{s(x)}\\
\leq\,&\,\frac{2}{\mu}(1-2R)+
\frac{1}{m}\int_{x_0-R}^{x_0+R}|x-x_0|^{-q}\,\dd x<\infty.
\end{split}
\end{equation*}
Hence, 
\begin{equation*}
\begin{split}
\int_0^1\left(\int_0^1|x-z|^{-\delta(x)}\,\dd z\right)\,\dd x<\infty.
\end{split}
\end{equation*}
Therefore, by Tonelli's Theorem, we have that $(x,z)\mapsto |x-z|^{-\delta(x)}$ belongs to $L^1(\mathrm{Q}_2)$, which in turn implies that
\begin{equation*}
\int_0^1|x-z|^{-\delta(x)}\,\dd x<\infty,
\end{equation*}
for almost all $z\in(0,1)$, by  Fubini's Theorem. 

Minor changes in the above arguments yield the same conclusion for functions $s$ with a finite number of zeros and a local behavior as \eqref{Eq:Example-s} around each of them. 
\end{example}

Next, in Definition \ref{Def:FractSpace} we present a Sobolev space of functions where smoothness is spatially dependent and related to $s(\cdot)$. 
First, we introduce the required notation.

For $i=1,\hdots,N$, let $\varphi_i,\psi_i:\mathrm{Q}_{N+1}\to\mathbb{R}$ be given by
\begin{align*}
\varphi_i(x,\tau)=\,&\Phi_i(x,\tau)^{1-p'}\left(\int_{\min\{x_i,\tau\}}^{\max\{x_i,\tau\}}\Phi_i(x_{t'}^i,\tau)^{1-p'}\,\dif t'\right)^{-p},\\[0.2cm]
\psi_i(x,\tau)=\,&\Phi_i(x,\tau)^{1-p'}\left(\int_{\min\{x_i,\tau\}}^{\max\{x_i,\tau\}}\Phi_i(x,\tau')^{1-p'}\,\dif\tau'\right)^{-p},
\end{align*}
where 
\begin{equation*}
	\Phi_i(x,\tau)=\mathrm{G}_s(x)|x_i-\tau|^{1-2s(x)},
\end{equation*}
and the notation $x_a^i$ for $a\in(0,1)$ means that the $i$th-coordinate of $x=(x_1,\hdots,x_N)\in\mathrm{Q}_N$ is replaced by $a$, that is:
\begin{equation*}
x_a^i=(x_1,\hdots,x_{i-1},a,x_{i+1},\hdots,x_N).
\end{equation*}

\begin{definition}\label{Def:FractSpace}
The space $\mathbb{W}^{s(\cdot),p}(\mathrm{Q}_N,\tilde{w},w_1,\hdots,w_N)$ is defined by
\begin{equation}\label{Eq:FractSpace}
\mathbb{W}^{s(\cdot),p}(\mathrm{Q}_N,\tilde{w},w_1,\hdots,w_N)=\left\{
v\in L^p(\mathrm{Q}_N,\tilde{w}):A_i(v)<\infty\text{ for all }i=1,\hdots,N\right\},
\end{equation}
with the norm
\begin{equation}\label{Eq:FractNorm}
\|v\|_{\mathbb{W}^{s(\cdot),p}(\mathrm{Q}_N,\tilde{w},w_1,\hdots,w_N)}=\left(\|v\|^p_{L^p(\mathrm{Q}_N,\tilde{w})}+\sum_{i=1}^NA_i(v)\right)^{1/p},
\end{equation}
where
\begin{equation*}
A_i(v)\hspace*{-0,05cm}=\hspace*{-0,05cm}
\underbrace{\int_0^1\hspace*{-0,05cm}\hdots\hspace*{-0,05cm}
\int_0^1}_{(N-1)\text{-fold}}\hspace*{-0,05cm}\left(\int_0^1\hspace*{-0,05cm}\int_0^1\hspace*{-0,05cm}w_i(x_t^i,\tau)|v(x_t^i)-v(x_\tau^i)|^p\dif \tau\hspace*{-0,05cm}\dif t\hspace*{-0,05cm}\right)\hspace*{-0,05cm}\dif x_1\hdots\dif x_{i-1}\,\dif x_{i+1}\hdots\dif x_N,
\end{equation*}
and 
\begin{equation*}
	w_i=\min\{\varphi_i,\psi_i\} \quad \text{for} \quad i=1,\hdots,N.
\end{equation*}
\end{definition}

In order to address that $s(\cdot)$ controls locally the differential regularity of elements in \linebreak $\mathbb{W}^{s(\cdot),p}(\mathrm{Q}_N,\tilde{w},w_1,\hdots,w_N)$, consider the following. For $s\in(0,1)$, let $W^{s,p}(Q_N)$ be the fractional Sobolev space of order $s$, that is,
\begin{equation*}
W^{s,p}(Q_N)=\left\{v\in L^2(Q_N):\int_{Q_N} \int_{Q_N}\frac{|v(x)-v(y)|^p}{|x-y|^{N+ps}}\,\dd x\,\dd y<\infty\right\},
\end{equation*}
equipped with the norm
\begin{equation*}
\|v\|_{W^{s,p}(Q_N)}=\left(\|v\|_{L^p(Q_N)}^p+\int_{Q_N}\int_{Q_N}\frac{|v(x)-v(y)|^p}{|x-y|^{N+ps}}\,\dd x\,\dd y\right)^{1/p}.
\end{equation*}
If $p=2$, we have $\mathbb{H}^s(Q_N)=W^{s,2}(Q_N)$. Then, note the following lemma that can be found in \cite{Ne1993} (see also \cite{KuFuJo1977}).

\begin{lemma}\label{Lem:Nekvinda}
Let $-1<\varepsilon<p-1$. There exists a positive constant $c$ such that
\begin{equation*}
\|v\|_{W^{1-\frac{1+\varepsilon}{p},p}(\mathrm{Q}_N)}^p \leq c\,\left(\|v\|_{L^p(\mathrm{Q}_N)}^p+\sum_{i=1}^N\mathcal{A}_i(v)\right),
\end{equation*}
for every $v\in L^p(\mathrm{Q}_N)$ that satisfies $\mathcal{A}_i(v)<\infty$ for all $i=1,\hdots,N$, where
\begin{equation*}
\mathcal{A}_i(v)\hspace*{-0,05cm}:=\hspace*{-0,05cm}
\underbrace{\int_0^1\hspace*{-0,05cm}\hdots\hspace*{-0,05cm}
\int_0^1}_{(N-1)\text{-fold}}\hspace*{-0,05cm}\left(\int_0^1\hspace*{-0,05cm}\int_0^1\hspace*{-0,05cm}\frac{|v(x_t^i)-v(x_\tau^i)|^p}{|t-\tau|^{p-\varepsilon}}\dif \tau\hspace*{-0,05cm}\dif t\hspace*{-0,05cm}\right)\hspace*{-0,05cm}\dif x_1\hdots\dif x_{i-1}\,\dif x_{i+1}\hdots\dif x_N.
\end{equation*}
\end{lemma}

We now can show the relation between  $\mathbb{W}^{s(\cdot),p}(\mathrm{Q}_N,\tilde{w},w_1,\hdots,w_N)$ and the classical Sobolev spaces.

\begin{theorem}\label{Thm:SecondConstant-Extra}
If $s(\cdot)=s\in(0,1)$ constant, then 
\begin{equation*}
\mathbb{W}^{s(\cdot),p}(\mathrm{Q}_N,\tilde{w},w_1,\hdots,w_N)\hookrightarrow W^{1-\frac{2(1-s)}{p},p}(\mathrm{Q}_N).
\end{equation*}
\end{theorem}

\begin{proof}
Let $\delta:=1-2s$ and consider $t,\tau\in(0,1)$. Since $\delta(1-p')>-1$, we have: 
\begin{align*}
\int_{\min\{t,\tau\}}^{\max\{t,\tau\}}|t'-\tau|^{\delta(1-p')}\,\dif t'
=\int_{\min\{t,\tau\}}^{\max\{t,\tau\}}|t-\tau'|^{\delta(1-p')}\,\dif \tau'=\,&\frac{|t-\tau|^{1+\delta(1-p')}}{1+\delta(1-p')}.
\end{align*}
Then, a direct calculation yields:
\begin{align*}
\varphi_i(x,\tau)=\psi_i(x,\tau)=\,&\frac{\mathrm{G}_s(1+\delta(1-p'))^{p}}{|x_i-\tau|^{p-\delta}},
\end{align*}
for all $(x,\tau)\in\mathrm{Q}_{N+1}$ since $\mathrm{G}_s$ is constant by assumption \ref{H2}. 
Therefore, 
\begin{equation*}
A_i(v)\hspace*{-0,05cm}=\hspace*{-0,05cm}
C(p,s)\hspace*{-0,05cm}\underbrace{\int_0^1\hspace*{-0,05cm}\hdots\hspace*{-0,05cm}
\int_0^1}_{(N-1)\text{-fold}}\hspace*{-0,05cm}\left(\int_0^1\hspace*{-0,05cm}\int_0^1\hspace*{-0,05cm}\frac{|v(x_t^i)-v(x_\tau^i)|^p}{|t-\tau|^{p-\delta}}\dif \tau\hspace*{-0,05cm}\dif t\hspace*{-0,05cm}\right)\hspace*{-0,05cm}\dif x_1\hdots\dif x_{i-1}\,\dif x_{i+1}\hdots\dif x_N,
\end{equation*}
where $C(p,s)=\mathrm{G}_s(1+\delta(1-p'))^{p}$. In addition, we notice that $L^1(\mathrm{Q}_N,\tilde{w})=L^1(\mathrm{Q}_N)$ since $\tilde{w}$ is constant. Now the conclusion follows from  Lemma \ref{Lem:Nekvinda} with $\varepsilon=\delta$.
\end{proof}

\begin{remark}
	In light of the previous result, it seems that a more appropriate notation for \linebreak $\mathbb{W}^{s(\cdot),p}(\mathrm{Q}_N,\tilde{w},w_1,\hdots,w_N)$ would be $\mathbb{W}^{1-\frac{2(1-s(\cdot))}{p},p}(\mathrm{Q}_N,\tilde{w},w_1,\hdots,w_N)$. We avoid this for the sake of brevity.
\end{remark}

The following lemma is a key tool for the improvement of the result in Theorem \ref{Thm:FirstTrace}. The proof can be found in \cite[Sect. 2.6]{OpKu1990}. 

\begin{lemma}[\textsc{Weighted Hardy-type inequality}]\label{Lem:HardyIneq}
Let $\rho$ be a weight function defined in the interval $(a,b)$. If 
\begin{equation*}
\int_a^b\rho(t)^{1-p'}\dif t<\infty,
\end{equation*}
then 
\begin{equation}\label{Eq:Hardy}
\int_a^b\hat{\rho}(t)|f(t)|^p\dif t\leq 
C_H(p)\int_a^b\rho(t)|f'(t)|^p\dif t,\qquad\qquad\forall\, x\in(a,b),
\end{equation}
for all absolutely continuous functions $f$ in $(a,b)$ that satisfy $\lim_{t\to a^+}f(t)=0$, where
\begin{equation*}
\hat{\rho}(t)=\rho(t)^{1-p'}\left(\int_a^t\rho(\xi)^{1-p'}\dif \xi\right)^{-p},
\end{equation*}
and $C_H(p)=p^p/(p-1)^{p-1}$. 
\end{lemma}

Now we are in shape to prove the improvement of Theorem \ref{Thm:FirstTrace}.

\begin{theorem}[\textsc{Improved trace theorem}]\label{Thm:SecondTrace}
Provided that {\normalfont\ref{H1}} to {\normalfont\ref{H5}} hold true, there exists a unique bounded linear operator
\begin{equation*}
\mathrm{tr}_{\mathrm{Q}_N}\,:\,\mathscr{H}_{0,L}^{1,p}(\mathcal{C},w)\to \mathbb{W}^{s(\cdot),p}(\mathrm{Q}_N,\tilde{w},w_1,\hdots,w_N),
\end{equation*}
that satisfies $\mathrm{tr}_{\mathrm{Q}_N} u=u(\,\cdot\,,0)$ for all $u\in \mathscr{H}_{0,L}^{1,p}(\mathcal{C},w)\cap C^\infty_c(\mathcal{C}_{\mathrm{Q}_N})$.

The same statement is true if we replace $\mathscr{H}_{0,L}^{1,p}(\mathcal{C},w)$ by the space $\mathscr{H}_{0,L}^{1,p}(\mathcal{C}^\tau,w)$, for every $\tau\geq 1$.

\end{theorem}

\begin{proof}

\noindent For the sake of simplicity, we give the proof only for $N=1$; with the natural changes, 
the proof adapts straightforward to the case $N\geq 2$.

Let $u\in \mathscr{H}_{0,L}^{1,p}(\mathcal{C},w)\cap C^\infty_c(\mathcal{C}_{\mathrm{Q}_1})$. Initially, we write:
\begin{align}\label{Proof:SecTrace-0}
A_1(u(\,\cdot\,,0))=\,&I_1+I_2,
\end{align}
where:
\begin{align*}
I_1:=\,&\int_0^1\int_0^tw_1(t,\tau)|u(t,0)-u(\tau,0)|^p\,\dif \tau\,\dif t,\\
I_2:=\,&\int_0^1\int_t^1w_1(t,\tau)|u(t,0)-u(\tau,0)|^p\,\dif \tau\,\dif t,
\end{align*}
where $w_1=\min\{\varphi_1,\psi_1\}$ as in Definition \ref{Def:FractSpace}. Next, we shall estimate $I_1$ and $I_2$ separately. For this, we introduce the auxiliary function $v:\mathrm{Q}_2\to\mathbb{R}$ given by
\begin{equation*}
v(t,\tau)=u(t,\max\{t,\tau\}-\min\{t,\tau\}).
\end{equation*}

We have:
\begin{align*}
I_1=\,&\int_0^1\int_0^tw_1(t,\tau)|v(t,t)-v(\tau,\tau)|^p\,\dif \tau\,\dif t\\
=\,&\int_0^1\int_0^tw_1(t,\tau)\left|\int_\tau^tD_1v(t',\tau)\,\dif t'+\int_\tau^tD_2v(t,\tau')\,\dif \tau'\right|^p\,\dif \tau\,\dif t\\
\leq\,&2^{p-1}\int_0^1\int_0^tw_1(t,\tau)\left|\int_\tau^tD_1v(t',\tau)\,\dif t'\right|^p\,\dif \tau\,\dif t\\
&+2^{p-1}\int_0^1\int_0^tw_1(t,\tau)\left|\int_\tau^tD_2v(t,\tau')\,\dif \tau'\right|^p\,\dif \tau\,\dif t,
\end{align*}
where $D_1v$ and $D_2v$ denote the partial derivative of $v$ with respect to the first and second coordinates, respectively. 

Interchanging the order of integration in the first term of the right hand side of the above inequality, and introducing the 
change of variable $\tilde{\tau}=-\tau$ in the second one, we find:
\begin{align}\label{Proof:SecTrace-1}
I_1\leq\,&2^{p-1}\int_0^1\int_\tau^1w_1(t,\tau)|f_1(t,\tau)|^p\,\dif t\,\dif \tau
+2^{p-1}\int_0^1\int_{-t}^0w_1(t,-\tilde{\tau})|f_2(t,\tilde{\tau})|^p\,\dif \tilde{\tau}\,\dif t,
\end{align}
where:
\begin{align*}
f_1(t,\tau)=\,\int_\tau^tD_1v(t',\tau)\,\dif t', & & \text{and} & & f_2(t,\tilde{\tau})=\,\int_t^{-\tilde{\tau}}D_2v(t,\tau')\,\dif \tau'.
\end{align*}

The function $f_1(\,\cdot\,,\tau)$ is absolutely continuous in $(\tau,1)$ and satisfies $\lim_{t\to\tau^+} f_1(t,\tau)=0$ for almost all $\tau\in(0,1)$. Additionally, by definition we observe that
\begin{align*}
\varphi_1(t,\tau)=\Phi_1(t,\tau)^{1-p'}\left(\int_\tau^t\Phi_1(t',\tau)^{1-p'}\,\dif t'\right)^{-p},\qquad\qquad\forall\,t\geq\tau,
\end{align*}
for almost all $\tau\in(0,1)$. Then, by Lemma \ref{Lem:HardyIneq}, we have:
\begin{equation}\label{Proof:SecTrace-2}
\int_\tau^1\varphi_1(t,\tau)|f_1(t,\tau)|^p\,\dif t\leq
C_H(p)\int_\tau^1\Phi_1(t,\tau)\left|D_1v(t,\tau)\right|^p\,\dif t,
\end{equation}
for almost all $\tau\in(0,1)$.

Similarly, the function $f_2(t,\,\cdot\,)$ is absolutely continuous in $(-t,0)$ and satisfies $\lim_{\tilde{\tau}\to-t^+}f_2(t,\tilde{\tau})=0$ for almost all $t\in(0,1)$. Since
\begin{align*}
\psi_1(t,-\tilde{\tau})
=\,&\Phi_1(t,-\tilde{\tau})^{1-p'}\left(\int_{-t}^{\tilde{\tau}}\Phi_1(t,-\tilde{\tau}')^{1-p'}\,\dif\tilde{\tau}'\right)^{-p},
\qquad\forall\,\tilde{\tau}\geq-t,
\end{align*}
for almost all $t\in(0,1)$, it follows by Lemma \ref{Lem:HardyIneq} that
\begin{equation}\label{Proof:SecTrace-3}
\int_{-t}^0\psi_1(t,-\tilde{\tau})|f_2(t,\tilde{\tau})|^p\,\dif \tilde{\tau}\leq
C_H(p)\int_{-t}^0\Phi_1(t,-\tilde{\tau})\left|D_2v(t,-\tilde{\tau})\right|^p\,\dif \tilde{\tau},
\end{equation}
for almost all $t\in(0,1)$.

Then, since $w_1=\min\{\varphi_1,\psi_1\}$, the estimation \eqref{Proof:SecTrace-1} in conjunction with \eqref{Proof:SecTrace-2} and \eqref{Proof:SecTrace-3} yields:
\begin{align*}
I_1
\leq\,&C_H(p)\,2^{p-1}\int_0^1\int_\tau^1\Phi_1(t,\tau)\left|D_1v(t,\tau)\right|^p\dif t\,\dif \tau\\
&+C_H(p)\,2^{p-1}\int_0^1\int_{-t}^0\Phi_1(t,-\tilde{\tau})\left|D_2v(t,-\tilde{\tau})\right|^p\dif \tilde{\tau}\,\dif t.
\end{align*}
Interchanging the order of integration in the first term of the r.h.s., and making the change of variable $\tau=-\tilde{\tau}$ in the second one, we obtain:
\begin{align}\label{Proof:SecTrace-4}
I_1
\leq\,&C_H(p)\,2^{p-1}\int_0^1\int_0^t\Phi_1(t,\tau)\left(\,\left|D_1v(t,\tau)\right|^p+\left|D_2v(t,\tau)\right|^p\,\right)\dif \tau\,\dif t.
\end{align}
Since the function $v$ is given by $v(t,\tau)=u(t,t-\tau)$ for $t>\tau$, we have: 
\begin{align*}
D_1v(t,\tau)=\,&
D_1u(t,t-\tau)+D_2u(t,t-\tau),\\
D_2v(t,\tau)=\,&
-D_2u(t,t-\tau).
\end{align*}
Then, 
\begin{align*}
\left|D_1v(t,\tau)\right|^p+\left|D_2v(t,\tau)\right|^p
\leq\,&
\left(\,\left|D_1u(t,t-\tau)\right|+\left|D_2u(t,t-\tau)\right|\,\right)^p+\left|D_2u(t,t-\tau)\right|^p\\
\leq\,& (2^{p-1}+1)\left(\,\left|D_1u(t,t-\tau)\right|^p+\left|D_2u(t,t-\tau)\right|^p\,\right)\\
\leq\,& 2^{p/2}(2^{p-1}+1)|\nabla u(t,t-\tau)|^p.
\end{align*}
Using this estimation in \eqref{Proof:SecTrace-4} and then making the change of variable $y=t-\tau$ in the inner integral, we find:
\begin{align*}
I_1
\leq\,&C_H(p)\,2^{p-1}2^{p/2}(2^{p-1}+1)\int_0^1\int_0^t\Phi_1(t,\tau)|\nabla u(t,t-\tau)|^p\dif \tau\,\dif t\\
=\,&C_H(p)\,2^{p-1}2^{p/2}(2^{p-1}+1)\int_0^1\int_0^t\Phi_1(t,t-y)|\nabla u(t,y)|^p\dif y\,\dif t.
\end{align*}
Hence,
\begin{align}\label{Proof:SecTrace-5}
I_1
\leq\,&C_H(p)\,2^{p-1}2^{p/2}(2^{p-1}+1)\int_0^1\int_0^1w(t,y)|\nabla u(t,y)|^p\dif y\,\dif t.
\end{align}

To estimate $I_2$, we first write:
\begin{align*}
I_2=\,&\int_0^1\int_t^1w_1(t,\tau)|v(t,t)-v(\tau,\tau)|^p\,\dif \tau\,\dif t\\
=\,&\int_0^1\int_0^\tau w_1(t,\tau)|v(t,t)-v(\tau,\tau)|^p\,\dif t\,\dif \tau,
\end{align*}
and notice that, in general, $I_2\neq I_1$ since $w_1(t,\tau)\neq w_1(\tau,t)$ for $s(\cdot)$ not constant. However, similarly as we obtained \eqref{Proof:SecTrace-5}, we identify the same bound for $I_2$:
\begin{align}\label{Proof:SecTrace-6}
I_2
\leq\,&C_H(p)\,2^{p-1}2^{p/2}(2^{p-1}+1)\int_0^1\int_0^1w(t,y)|\nabla u(t,y)|^p\dif y\,\dif t.
\end{align}

Using \eqref{Proof:SecTrace-5} and \eqref{Proof:SecTrace-6} in  \eqref{Proof:SecTrace-0}, we obtain:
\begin{align*}
A_1(u(\,\cdot\,,0))
\leq\,&C_H(p)\,2^p\,2^{p/2}(2^{p-1}+1)\|\nabla u\|_{L^p(\mathcal{C},w)}^p,
\end{align*}
hence,
\begin{align*}
A_1(u(\,\cdot\,,0))
\leq\,&C_H(p)\,2^p\,2^{p/2}(2^{p-1}+1)\|u\|_{W^{1,p}(\mathcal{C},w)}^p.
\end{align*}

In addition, we know by the Theorem \ref{Thm:FirstTrace}  that:
\begin{equation*}
\|u(\,\cdot\,,0)\|_{L^p(\mathrm{Q}_1,\tilde{w})}^p\leq (1+p)^p\sigma^{-2p/p'}\|u\|_{W^{1,p}(\mathcal{C},w)}^p,
\end{equation*} 
where $\sigma$ is some arbitrary, but fixed, number in $(0,1)$.

Therefore, $u(\,\cdot\,,0)\in \mathbb{W}^{s(\cdot),p}(\mathrm{Q}_1,\tilde{w},w_1)$ and
\begin{equation*}
\|u(\,\cdot\,,0)\|_{\mathbb{W}^{s(\cdot),p}(\mathrm{Q}_1,\tilde{w},w_1)}\leq C(p,\sigma)\|u\|_{W^{1,p}(\mathcal{C},w)},
\end{equation*}
where $C(p,\sigma)=(C_H(p)\,2^p\,2^{p/2}(2^{p-1}+1)+(1+p)^p\sigma^{-2p/p'})^{1/p}$.

The operator $\mathrm{tr}_{\mathrm{Q}_1}$ is the unique bounded linear extension of the map $u\mapsto u(\,\cdot\,,0)$ to $\mathscr{H}_{0,L}^{1,p}(\mathcal{C},w)$. 

The proof when $\mathscr{H}_{0,L}^{1,p}(\mathcal{C},w)$ is replaced by $\mathscr{H}_{0,L}^{1,p}(\mathcal{C}^\tau,w)$ where $\tau\geq 1$ is identical.
\end{proof}

\begin{remark}[Surjectivity of trace operator]\label{rem:surj}
Although the previous result represents an improvement on the $\Omega$-trace characterization for functions in $\mathscr{H}^{1,p}_{0,L}(\mathcal{C},w)$, nothing can be said about the surjectivity of the trace operator $\tr:\mathscr{H}^{1,p}_{0,L}(\mathcal{C},w)\to \mathbb{W}^{s(\cdot),p}(Q_N,w,w_1,\hdots,w_N)$ for $s(\cdot)$ non-constant.
\end{remark}

{\begin{remark}
If $s(\cdot)=s\in(0,1)$ is constant, then it follows by Theorem \ref{Thm:SecondConstant-Extra} that \linebreak $\mathbb{W}^{s,2}(\mathrm{Q}_N,\tilde{w},w_1,\hdots,w_N)\hookrightarrow W^{1-\frac{2(1-s)}{p},p}(\mathrm{Q}_N)$. Hence, the trace result in Theorem \ref{Thm:FirstTrace} is again in accordance to the classical case, see \cite[Theorem 2.8]{Ne1993} (see also Remark \ref{RK:Classic1}). Moreover, if $p=2$ and $s\in(0,1/2)$, we observe that 
\begin{equation}
\mathbb{W}^{s,2}(\mathrm{Q}_N,\tilde{w},w_1,\hdots,w_N)\hookrightarrow H,
\end{equation}
since $\mathbb{H}^s(\mathrm{Q}_N)=\mathbb{H}_0^s(\mathrm{Q}_N)=H$, so in this case we further partially recover the trace result in \cite[Lemma 2.2]{CaEtAl2011} given that $\mathscr{H}^{1,p}_{0,L}(\mathcal{C},y^{1-2s})\subset H^1_{0,L}(\mathcal{C},y^{1-2s})$.  
\end{remark}}


\section{Cases where the Poincar\'{e} inequality holds}\label{Sect:PoincareIneq}

We address now in this section cases and conditions on  $s(\cdot)$ not constant that are sufficient for the Poincar\'e inequality to hold true. Two results are given, one in the entire cylinder and one in the truncated {cylinder}; see  Theorem \ref{Thm:PoincareStep} and Theorem \ref{Thm:PoincareStepBis} respectively. From now on until the end of the section, we assume that $\mathrm{G}_s=\mathrm{G}_s^{(1)}$ constant, see \eqref{Eq:Gexamples} in Example \ref{Ex:sG}; and $s(\cdot)$ is given by
\begin{equation}\label{Eq:Step-s}
s(\cdot)=\sum_{i=1}^Ms_i\mathds{1}_{\Omega_i}(\cdot),
\end{equation}
where $s_i\in(0,1)$ for $i=1,\hdots,M$ and $\{\Omega_i:i=1,\hdots,M\}$ is a finite collection of non-empty open subsets of $\Omega$ that satisfies $\bigcup_{i=1}^M\bar{\Omega}_i=\bar{\Omega}$.
In other words, we assume that $s(\cdot)$ is a step function in $\Omega$ with range contained in the interval $(0,1)$. Our first example is given by the next theorem which basically states that the Poincar\'{e} inequality holds provided that all pieces $\Omega_i$ of the partition of $\Omega$ touch the boundary $\partial\Omega$.

\begin{theorem}\label{Thm:PoincareStep}
Assume that $\mathrm{G}_s=\mathrm{G}_s^{(1)}$ constant and $s(\cdot)$ is given by \eqref{Eq:Step-s}. If
\begin{equation}\label{Eq:Touch}
|\partial{\Omega}_i\cap\partial{\Omega}|>0,\quad\qquad\forall\,i=1,\hdots,M,
\end{equation}
then there exists a positive constant $C_P(p,\Omega_1,\hdots,\Omega_M)$ that satisfies 
\begin{equation}\label{Eq:PoincareStep}
\|u\|_{L^p(\mathcal{C},w)}\leq C_P(p,\Omega_1,\hdots,\Omega_M)\|\nabla u\|_{L^p(\mathcal{C},w)},
\end{equation}
for all $u\in\mathscr{H}^{1,p}_{0,L}(\mathcal{C},w)$.
\end{theorem}

\begin{proof}
The proof is quite direct, thanks to the existence of a Poincar\'e inequality for functions in $C^\infty(\Omega_i)$ that vanish on a subset of non-zero measure of $\partial\Omega_i$: Let $u\in C_c^\infty(\mathcal{C}_\Omega)\cap W^{1,p}(\mathcal{C},w)$ and $i\in\{1,\hdots,M\}$. For every $y>0$, the function $u(\,\cdot\,,y)$ belongs to $C^\infty(\Omega_i)$ and vanishes on a portion with non-zero measure of $\partial\Omega_i$, by \eqref{Eq:Touch}. Then, by the Poincar\'e inequality, we have
\begin{equation}\label{Proof:PoincareStep-1}
\int_{\Omega_i}|u(x,y)|^p\,\dd x\leq c_i\int_{\Omega_i}|\nabla_x u(x,y)|^p\,\dd x,
\end{equation}
where $c_i$ is a positive constant that depends only on $\Omega_i$ and $p$, and $\nabla_x u$ is the gradient of $u$ with respect to the first $N$ coordinates.
Multiplying \eqref{Proof:PoincareStep-1} by $y^{1-2s_i}$, then integrating for $y\in(0,\infty)$, and finally adding up for $i=1,\hdots,M$, we obtain
\begin{equation*}
\int_{\mathcal{C}}y^{1-2s(x)}|u(x,y)|^p\,\dd X\leq c\int_{\mathcal{C}}y^{1-2s(x)}|\nabla_x u(x,y)|^p\,\dd X,
\end{equation*}
where $c=c_1+\hdots+c_M$. Since $|\nabla_xu|^p\leq|\nabla u|^p$ in $\mathcal{C}$ and $\mathrm{G}_s$ is constant, we get
\begin{equation*}
\int_{\mathcal{C}}w(x,y)|u(x,y)|^p\,\dd X\leq c\int_{\mathcal{C}}w(x,y)|\nabla u(x,y)|^p\,\dd X,
\end{equation*}
for all $u\in C_c^\infty(\mathcal{C}_\Omega)\cap W^{1,p}(\mathcal{C},w)$. Now \eqref{Eq:PoincareStep} follows by density.
\end{proof}

Next we prove that the truncated domain allows a much more amenable result than the one in the complete cylinder $\mathcal{C}$. In particular, we prove that \eqref{Eq:Step-s} is a sufficient condition for the Poincar\'{e} inequality to hold; the result is given in next Theorem \ref{Thm:PoincareStepBis}. The proof requires the following auxiliary lemma, see \cite[Theorem 5.2]{Ku1980} for its proof.

\begin{lemma}[\textsc{Classical Hardy inequality}]\label{Lem:ClassicalHardyIneq}
Let $\varepsilon>p-1$ and let $f$ be a differentiable function almost everywhere in $(0,\infty)$ that satisfies $\lim_{t\to \infty}f(t)=0$. If
\begin{equation*}
\int_0^\infty t^\varepsilon|f'(t)|^p\,\dif t<\infty,
\end{equation*}
then
\begin{equation*}
\int_0^\infty t^{\varepsilon-p}|f(t)|^p\,\dif t\leq C_H(p,\varepsilon)\int_0^\infty t^\varepsilon|f'(t)|^p\,\dif t<\infty,
\end{equation*}
where $C_H(p,\varepsilon)=p^p/(\varepsilon-p+1)^p$.
\end{lemma}

We are now in a position to present the final result in this section.

\begin{theorem}\label{Thm:PoincareStepBis}
Assume that $\mathrm{G}_s=\mathrm{G}_s^{(1)}$ constant and $s(\cdot)$ is given by \eqref{Eq:Step-s}. For every $\tau>0$ there exists a positive constant $C_P(\tau,p,\Omega_1,\hdots,\Omega_M)$ that satisfies 
\begin{equation}\label{Eq:PoincareStepBis}
\|u\|_{L^p(\mathcal{C}^\tau,w)}\leq C_P(\tau,p,\Omega_1,\hdots,\Omega_M)\|\nabla u\|_{L^p(\mathcal{C}^\tau,w)},
\end{equation}
for all $u\in\mathscr{H}^{1,p}_{0,L}(\mathcal{C}^\tau,w)$.
\end{theorem}

\begin{proof}
Let $\tau>0$ and $u\in C_c^\infty(\mathcal{C}_{\Omega}^\tau)\cap W^{1,p}(\mathcal{C}^\tau,w)$. Initially, we write:
\begin{align}\label{Proof:Poincare-1}
\int_{\mathcal{C}^\tau} y^{1-2s(x)}|u(x,y)|^p\,\dif X=\sum_{i=1}^MI_i,
\end{align}
where
\begin{align*}
I_i:=\,&\int_0^\tau y^{1-2s_i}\int_{\Omega_i} |u(x,y)|^p\,\dif x\,\dif  y.
\end{align*}

We denote by $c$ a positive constant that may depend only on $p$ and the partition $\{\Omega_i:i=1,\hdots,M\}$, whose numerical value may be different from one line to another.

Let $i\in\{1,\hdots,M\}$. We define
\begin{equation*}
\bar{u}_i(y)=\frac{1}{|\Omega_i|}\int_{\Omega_i}u(x,y)\,\dif x,
\end{equation*} 
and observe that 
\begin{align}\label{Proof:Poincare-2}
I_i 
\leq c\,(I_{i1}+I_{i2}),
\end{align}
where
\begin{align*}
I_{i1}:=\,&\int_0^\tau y^{1-2s_i}\int_{\Omega_i} |u(x,y)-\bar{u}_i(y)|^p\,\dif x\,\dif y,\\
I_{i2}:=\,&\int_0^\tau y^{1-2s_i}\int_{\Omega_i} |\bar{u}_i(y)|^p\,\dif x\,\dif y=
|\Omega_i|\int_0^\tau y^{1-2s_i}|\bar{u}_i(y)|^p\,\dif y.
\end{align*}
 
For each fixed $y\in(0,\tau)$, the function $u(\,\cdot\,,y)$ belongs to $C^\infty(\Omega_i)$. Thus, by the Poincar\'e-Wirtinger's inequality, we obtain:
\begin{align*}
\int_{\Omega_i} |u(x,y)-\bar{u}_i(y)|^p\,dx\leq c\,\int_{\Omega_i}|\nabla_x u(x,y)|^p\,\dif x.
\end{align*}
From this, similarly as in the proof of Theorem \ref{Thm:PoincareStep}, we find:
\begin{align}\label{Proof:Poincare-3}
I_{i1}\leq\,&c\int_0^\tau \int_{\Omega_i} y^{1-2s_i}|\nabla u(x,y)|^p\,\dif x\,\dif y.
\end{align}

Let $\bar{u}_{ext}$ be the extension by zero of $\bar{u}$ to $[0,\infty)$. Notice that $\bar{u}_{ext}$ is differentiable almost everywhere in $(0,\infty)$ since $u(x,\,\cdot\,)\in C^\infty([0,\tau])$ for all $x\in\Omega_i$, and, trivially, $\bar{u}_{ext}$ satisfies $\lim_{y\to \infty}\bar{u}_{ext}(y)=0$. Also, observe that
\begin{align*}
\int_0^\infty y^{1+p-2s_i}|\bar{u}'_{ext}(y)|^p\,\dif y=
\int_0^\tau y^{1+p-2s_i}|\bar{u}'(y)|^p\,\dif y\leq 
c\int_0^\tau y^{1+p-2s_i}\,\dif y<\infty,
\end{align*}
since $1+p-2s_i>-1$ and $\bar{u}'$ is bounded in $[0,\tau]$. 

Then, by the classical Hardy inequality in Lemma \ref{Lem:ClassicalHardyIneq} with $\varepsilon=1+p-2s_i$, we have:
\begin{equation}\label{Proof:Poincare-4}
I_{i2}=|\Omega_i|\int_0^\infty y^{\varepsilon-p}|\bar{u}_{ext}(y)|^p\,dy\leq c\,|\Omega_i|\int_0^\infty y^{\varepsilon}|\bar{u}'_{ext}(y)|^p\,\dif y.
\end{equation}
We now observe that:
\begin{align*}
\int_0^\infty y^{\varepsilon}|\bar{u}'_{ext}(y)|^p\dif y=\,&
\frac{1}{|\Omega_i|^p}\int_0^\tau y^{\varepsilon}\left|\int_{\Omega_i} D_{N+1}u(x,y)\,dx\right|^p\dif y\\
\leq\,&
\frac{1}{|\Omega_i|^p}\int_0^\tau y^{\varepsilon}\left(\int_{\Omega_i} |\nabla u(x,y)|\,dx\right)^p\dif y,
\end{align*}
where $D_{N+1}u$ is the partial derivative of $u$ with respect to the $(N+1)$ coordinate. Then, by the H\"older's inequality on the inner integral, we have:
\begin{align*}
\int_0^\infty y^{\varepsilon}|\bar{u}'_{ext}(y)|^p\dif y\leq\,&
\frac{1}{|\Omega_i|}\int_0^\tau y^{\varepsilon}\int_{\Omega_i} |\nabla v(x,y)|^p\,\dif x\,\dif y. 
\end{align*}
With this estimation in \eqref{Proof:Poincare-4}, we find:
\begin{align}\label{Proof:Poincare-5}
I_{i2}\leq\,&c\,\tau^p\int_0^\tau \int_{\Omega_i} y^{1-2s_i}|\nabla u(x,y)|^p\,\dd x\,\dd y.
\end{align}

Finally, using \eqref{Proof:Poincare-3} and \eqref{Proof:Poincare-5} in \eqref{Proof:Poincare-2}, we obtain:
\begin{equation*}
I_i\leq c\,\tau^p\int_0^\tau \int_{\Omega_i} y^{1-2s_i}|\nabla u(x,y)|^p\,\dif x\,\dif y,
\end{equation*}
and hence, by \eqref{Proof:Poincare-1} and since $\mathrm{G}_s$ is constant, we have:
\begin{equation*}
\int_{\mathcal{C}^\tau}\hspace*{-0.2cm} w(x,y)|u(x,y)|^p\,\dif X
\leq
 c\,\tau^p\hspace*{-0.1cm}\int_{\mathcal{C}^\tau}\hspace*{-0.2cm} w(x,y)|\nabla u(x,y)|^p\,\dif X,
\end{equation*}
for all $u\in C_c^\infty(\mathcal{C}^\tau_\Omega)\cap W^{1,p}(\mathcal{C}^\tau,w)$. Then we obtain \eqref{Eq:PoincareStepBis} by density. 
\end{proof}

\section{Second definition and solution to $(-\Delta)^{s(\cdot)}v=h$}\label{Sect:Elliptic}
We are now in a position to give a new definition for the operator $(-\Delta)^{s(\cdot)}$, and to solve the associated Poisson problem for right hand sides defined on $\Omega$. The arguments below are very similar to those developed in Section \ref{Sect:Abstract} but now we assume some extra conditions on the function $s(\cdot)$ and the domain $\Omega$, which enable a better characterization of the domain of $(-\Delta)^{s(\cdot)}$. We present the ideas for the semi-infinite cylinder $\mathcal{C}$, but the same arguments are valid for a truncated one $\mathcal{C}^\tau$.

From now on, we assume that the functions $s(\cdot)$ and $\mathrm{G}_s$ satisfy hypotheses \ref{H1}, \ref{H2}, and \ref{H3}. Further, we assume that the Poincar\'{e} inequality holds true, that is there exists $C>0$ such that
\begin{equation*} 
\|u\|_{L^2(\mathcal{C},w)}\leq C \|\nabla u\|_{L^2(\mathcal{C},w)},\qquad\forall\,u\in\mathscr{H}^{1,p}_{0,L}(\mathcal{C},w). 
\end{equation*} 
For example, this is satisfied under the assumptions of Theorem \ref{Thm:PoincareStep} (see Theorem \ref{Thm:PoincareStepBis} for the case of a truncated cylinder). In particular, this implies that
\begin{equation*}
\mathscr{H}_{0,L}^{1,2}(\mathcal{C},w)= \mathscr{L}_{0,L}^{1,2}(\mathcal{C},w),
\end{equation*}
algebraically and topologically. We endow $\mathscr{H}_{0,L}^{1,2}(\mathcal{C},w)$ with the norm $\|v\|_{\mathscr{H}_{0,L}^{1,p}(\mathcal{C},w)}:=\|\nabla v\|_{L^2(\mathcal{C},w)}$.
Under the hypotheses  assumed, we have established in Theorem \ref{Thm:FirstTrace} an $\Omega$-trace operator 
\begin{equation}\label{eq:traza}
	\tr:\mathscr{H}_{0,L}^{1,2}(\mathcal{C},w)\to L^2(\Omega,\tilde{w}),
\end{equation}
and proved it is bounded, linear, and  such that $\tr u=u(\,\cdot\,,0)$ for all $u\in W^{1,p}(\mathcal{C},w)\cap C^\infty_c(\mathcal{C}_\Omega)$. Note that this operator is not, however, surjective. Subsequently, consider
\begin{equation*}
	\mathscr{W}_{0}^{1,2}(\mathcal{C},w):=\{u\in \mathscr{H}_{0,L}^{1,2}(\mathcal{C},w): \tr u=0\},
\end{equation*}
which is a closed subspace of $\mathscr{H}_{0,L}^{1,2}(\mathcal{C},w)$. Hence, a space of abstract traces on $\Omega$ of functions in $\mathscr{H}_{0,L}^{1,2}(\mathcal{C},w)$ can be defined as the quotient space 
\begin{equation*}
\mathscr{Y}(\Omega,w):=\mathscr{H}_{0,L}^{1,2}(\mathcal{C},w)/\mathscr{W}_0^{1,2}(\mathcal{C},w).
\end{equation*}
\begin{remark}
Due to the absence of density results of the type ``$H=W$'' for non-Muckenhoupt weights, we are not in a position to assure that the spaces $\mathscr{X}(\Omega,w)$ and $\mathscr{Y}(\Omega,w)$ are actually the same.
\end{remark}

Immediately from here, via the isomorphism theorems, we can argue that there is an isomorphism 
\begin{equation}\label{Eq:Iso}
\varphi:\mathscr{Y}(\Omega,w)\xrightarrow{\sim} \tr \mathscr{H}_{0,L}^{1,2}(\mathcal{C},w).
\end{equation}
Moreover, one can simply consider $\varphi$ to be given by $[u]\mapsto \tr u$. However, in order to identify $\mathscr{Y}(\Omega,w)$ with a subset of functions defined on $\Omega$, we need further information related with the structure of the function space $\tr \mathscr{H}_{0,L}^{1,2}(\mathcal{C},w)$.

Analogously as in Section \ref{Sect:Abstract}, we define 
\begin{equation}
	\TR :\mathscr{H}_{0,L}^{1,2}(\mathcal{C},w)\to \mathscr{Y}(\Omega,w),
\end{equation}
as $\TR u:=[u]$, and observe that $\TR$ is surjective by definition. In this setting we identify the abstract $\Omega$-trace of $u\in \mathscr{H}_{0,L}^{1,2}(\mathcal{C},w)$  with the equivalence class $[u]$ that contains $u$. The space $\mathscr{Y}(\Omega,w)$ is then endowed  with the usual quotient norm
\begin{align*}
\|\TR u\|_{\mathscr{Y}(\Omega,w)}=\|[u]\|_{\mathscr{Y}(\Omega,w)}:=\,&\inf\{ \|u-z\|_{\mathscr{H}_{0,L}^{1,2}(\mathcal{C},w))}:z\in \mathscr{W}_0^{1,2}(\mathcal{C},w))\}.
\end{align*}
As before, we have $\TR \mathscr{L}_{0,L}^{1,2}(\mathcal{C},w)=\mathscr{Y}(\Omega,w)$. Note that $\mathscr{Y}(\Omega,w)$ is a Hilbert space, given that $\mathscr{H}_{0,L}^{1,2}(\mathcal{C},w)$ and $\mathscr{H}_0^{1,2}(\mathcal{C},w)$ are also Hilbert spaces.

Identically as in Theorem \ref{Thm:AbstractMinim}, we argue the existence of the weighted harmonic extension operator 
\begin{equation*}
S: \TR \mathscr{H}_{0,L}^{1,2}(\mathcal{C},w) \to  \mathscr{H}_{0,L}^{1,2}(\mathcal{C},w),\qquad
v\mapsto S(v)=u.
\end{equation*}
where $u$ is the solution to 
\begin{equation*}
\begin{split}
&\mathrm{minimize}\quad J(u)\quad \mathrm{over}\quad  \mathscr{H}_{0,L}^{1,2}(\mathcal{C},w), \\
&\mathrm{subject\:\:to}\quad \TR u=v,
\end{split}
\end{equation*}
for
\begin{equation*}
J(u):=\frac{1}{2}\|u\|^2_{\mathscr{H}_{0,L}^{1,2}(\mathcal{C},w)}=\frac{1}{2}\int_{\mathcal{C}}w|\nabla u|^2\dif X.
\end{equation*}
The well-posedness of the map $S$ allows us to establish a definition for the fractional Laplacian with spatially variable order.
   
\begin{definition}\label{Def:AbstractLaplacian2}
Let $\mathscr{Y}(\Omega,w)'$ be the dual space of $\mathscr{Y}(\Omega,w)$. The operator 
\begin{equation*}
	(-\Delta)^{s(\cdot)}:\mathscr{Y}(\Omega,w)\to\mathscr{Y}(\Omega,w)',
\end{equation*} is determined as follows: for $v \in \mathscr{Y}(\Omega,w)$, then $(-\Delta)^{s(\cdot)}v\in \mathscr{Y}(\Omega,w)'$ is defined by 
\begin{equation*}\label{Eq:AbstractLaplacian2}
\langle(-\Delta)^{s(\cdot)}v, \TR\psi \rangle_{\mathscr{Y}',\mathscr{Y}}=\displaystyle\int_{\mathcal{C}}w\,\nabla S(v)\cdot\nabla \psi\,\dd X,\qquad\qquad\psi\in \mathscr{H}_{0,L}^{1,2}(\mathcal{C},w).
\end{equation*}
\end{definition}
Since Proposition \ref{prop:Lagrange} holds true with the usual changes, the operator is then well-defined and  Theorem \ref{Thm:AbstractElliptic} is also proven mutatis mutandis: For a $h\in \mathscr{Y}(\Omega,w)'$, the equation 
\begin{equation}\label{Eq:AbstractEq2}
(-\Delta)^{s(\cdot)}v=h \quad \mbox{in } \Omega,
\end{equation}
admits a unique solution $v\in \mathscr{Y}(\Omega,w)$ that is given by $v=\TR\,u$, where $u$ solves
\begin{align}\label{Eq:AbstractMinEllip2}
\mathrm{minimize}\quad \mathcal{J}(u)\quad \mathrm{over}\quad \mathscr{H}_{0,L}^{1,2}(\mathcal{C},w),
\end{align}
for 
\begin{equation*}
\mathcal{J}(u):=\frac{1}{2}\displaystyle\int_{\mathcal{C}}w\,|\nabla u|^2\,\dd X-\langle h, \Tr u\rangle_{\mathscr{Y}', \mathscr{Y}}.
\end{equation*}

Although this approach seems equivalent to the one in Section \ref{Sect:Abstract}, in this setting we have a more detailed representation of the elements $\mathscr{Y}(\Omega,w)$. In fact, within this approach, there exists an injection 
\begin{equation*}
I:\mathscr{Y}(\Omega,w)\to L^2(\Omega,\tilde{w}),\qquad
u\mapsto I([u])=\tr u,
\end{equation*}
which is linear and bounded. Linearity follows directly, and boundedness follows given that for arbitrary $z\in \mathscr{W}_{0}^{1,2}(\mathcal{C},w)$,
\begin{equation*}
\|I([u])\|_{L^2(\Omega,\tilde{w})}=\|\tr u \|_{L^2(\Omega,\tilde{w})}=\|\tr (u-z)\|_{L^2(\Omega,\tilde{w})}\leq C \|u-z\|_{\mathscr{H}_{0,L}^{1,2}(\mathcal{C},w)},
\end{equation*}
where we have used the linearity of $\tr$ and that $\tr z=0$, and then
\begin{equation*}
	\|I([u])\|_{L^2(\Omega,\tilde{w})}\leq C \inf_{y\in \mathscr{W}_{0}^{1,2}(\mathcal{C},w)}  \|u-z\|_{\mathscr{H}_{0,L}^{1,2}(\mathcal{C},w)}=C \|[u]\|_{\mathscr{Y}(\Omega,w)}.
\end{equation*}
In order to see that $I$ is an injection, suppose that $I([u])=0$, then $\tr u=0$ so that $u\in \mathscr{W}_{0}^{1,2}(\mathcal{C},w)$, and the class $\mathscr{W}_{0}^{1,2}(\mathcal{C},w)$ is the zero element of $\mathscr{Y}(\Omega,w)$. This identification allows us to consider $I$ to be the identity, and identify the continuous embedding 
\begin{equation*}
\mathscr{Y}(\Omega,w)\hookrightarrow L^2(\Omega,\tilde{w}).
\end{equation*} 
For a schematic relationship between the trace operators $\tr$, $\Tr$, the isomorphism $\varphi$ and the embedding $I$, see Figure \ref{fig:diagram}. 
An amenable consequence of this identification is given in Theorem \ref{Thm:AbstractElliptic2}, however in first place we address the reduction to case where $s(\cdot)=s\in (0,1)$, a constant, where we obtain that $H$ is recovered as the domain of $(-\Delta)^s$.

\begin{figure}[!h]
\centering\label{fig:diagram}
\begin{tikzpicture}
  \matrix (m)
    [
      matrix of math nodes,
      row sep    = 4em,
      column sep = 12em
    ]
    {
      \mathscr{H}_{0,L}^{1,2}(\mathcal{C},w)              &  \tr \mathscr{H}_{0,L}^{1,2}(\mathcal{C},w) 
\\
      \mathscr{Y}(\Omega,w) &  \begin{cases}
        L^2(\Omega,\tilde{w}), \qquad \text { if   \ref{H1}-\ref{H3}}\\ 
        \mathbb{W}^{s(\cdot),2}(\mathrm{Q}_N,\tilde{w},w_1,\hdots,w_N), \text { if   \ref{H1}-\ref{H5}}\\
        \end{cases}           \\
    };
  \path
    (m-1-1) edge [->] node [left] {$\TR$} (m-2-1)
    (m-1-1.east |- m-1-2)
      edge [->] node [above] {$\tr$} (m-1-2)
    (m-2-1.east) edge [yshift=.3cm, ->] node [below,rotate=20,xshift=.25cm] {$\varphi$} node [above,rotate=20,xshift=.25cm] {$\simeq$} (m-1-2)
    (m-2-1.east) edge [{Hooks[right,length=0.8ex]}->,yshift=-.10cm] node [below,yshift=.50cm] {$I$}  (m-2-2);
     (m-2-2.east) edge [->] node [right] {\hspace{2cm}$\bigcap$}  (m-1-2);
\end{tikzpicture}
\caption{Diagram relating the operators $\tr, \TR$, the isomorphism $\varphi$, and the operator $I$.}
\end{figure}

\begin{theorem}

\label{thm:surjectivitytrace}
Let $s(\cdot)=s\in(0,1)$ be constant and suppose that functions in $\mathscr{H}_{0,L}^{1,2}(\mathcal{C},w)$ satisfy a Poincar\'e inequality, then
\begin{equation*}
\tr \mathscr{H}_{0,L}^{1,2}(\mathcal{C},w)=H,
\end{equation*} 
and therefore, 
\begin{equation*}
\mathscr{Y}(\Omega,w)\simeq H.
\end{equation*}
\end{theorem}

\begin{proof}
Given that $s(\cdot)=s\in(0,1)$ is constant, we have that $\mathrm{G}$ is constant, and hence $\mathscr{H}_{0,L}^{1,2}(\mathcal{C},w)=\mathscr{H}_{0,L}^{1,2}(\mathcal{C},y^{1-2s})$. Additionally, by Remark \ref{RK:Classic1}, we have that $\tr \mathscr{H}_{0,L}^{1,2}(\mathcal{C},y^{1-2s})\subset H$. Then, there is  only left to prove that for each $v\in H$ there exists a sequence $\{u_n\}$ in $C^\infty_c(\mathcal{C}_\Omega)\cap W^{1,2}(\mathcal{C},y^{1-2s})$ convergent in the sense of $W^{1,2}(\mathcal{C},y^{1-2s})$ to a $u\in \mathscr{H}_{0,L}^{1,2}(\mathcal{C},y^{1-2s})$, and such that $\tr u= v$. We divide the proof into four steps for the sake of clarity.
	
	\textit{Step 1:} Let $v\in H$ be arbitrary. Since $H=\mathbb{H}_0^s(\Omega)$ for $s\in(0,1/2)$ or $s\in(1/2,1)$, and $H=\mathbb{H}_{00}^s(\Omega)$ for $s=1/2$, it follows that $C_c^\infty(\Omega)$ is dense in $H$. Then, there exists a sequence $\{v_k\}$ in $C_c^\infty(\Omega)$ such that 
	\begin{equation*}
		v_k\to v \qquad \text{ in } H,
	\end{equation*}
	as $k\to \infty$. We denote $v=\sum_{n=1}^\infty b_n\varphi_n$ and $v_k=\sum_{n=1}^\infty b^k_n\varphi_n$ to their spectral decomposition where $b^k_n\to b_n$ as $k\to\infty$. Further, define $u,u_k:\mathcal{C}\to\mathbb{R}$ by 
	\begin{equation*}
		u(x,y)=\sum_{n=1}^\infty b_n\varphi_n(x)g_n(y) \quad\text{and}\quad u_k(x,y)=\sum_{n=1}^\infty b^k_n\varphi_n(x)g_n(y),
	\end{equation*}
	where each $g_n$ satisfies the Bessel equation:
	\begin{align*}
		g_n''+\frac{1-2s}{y}g_n'-\lambda_kg_n&=0 \qquad \text{ in }\quad (0,+\infty),\\
		g_n(0)&=1,\\
		g_n(+\infty)&=0.
	\end{align*}
	Since the Poincar\'{e} inequality is valid for functions in $W^{1,2}(\mathcal{C},y^{1-2s})$, by the construction of the proof in \cite[Proposition 2.1]{CaEtAl2011}, we have that $u,u_n\in W^{1,2}(\mathcal{C},y^{1-2s})$, and
	\begin{equation*}
		\int_0^\infty\int_\Omega y^{1-2s}|\nabla u(x,y)-\nabla u_k(x,y)|^2\dif x \dif y=c_{N,s}\sum _{k=1}^{+\infty}(b_n-b_n^k)^2\mu_n^s=c_{N,s}\|v-v_k\|^2_H,
	\end{equation*}
	and thus
	\begin{equation}\label{eq:vnconv}
		u_k\to u \qquad \text{ in } W^{1,2}(\mathcal{C},y^{1-2s}),
	\end{equation}
	as $k\to \infty$. Note that since $v_k$ has compact support, the support of $u_k$ is uniformly away from $\partial_L \mathcal{C}$.
	
	\textit{Step 2:} For $\tau\geq 1$ and $0<\sigma<1$, we consider a smooth non-increasing function $\eta_\tau:\mathbb{R}^+\to [0,1]$ such that:
\begin{equation*}
\eta_\tau(y)=1\quad\text{if}\quad 0<y<\tau-\sigma,\qquad\qquad
\eta_\tau(y)=0\quad\text{if}\quad y>\tau,
\end{equation*}
and notice that the function $u_{k,\tau}(x,y):=\eta_\tau(y)u_k(x,y)$ belongs to $W^{1,2}(\mathcal{C},y^{1-2s})$. By direct calculation we have that 
	\begin{equation}\label{eq:vntconv}
		u_{k,\tau}\to u_n \qquad \text{ in } W^{1,2}(\mathcal{C},y^{1-2s}),
	\end{equation}
	as $\tau\to \infty$.
	
	\textit{Step 3:}  For $0<\varepsilon\ll 1$ and $\tau'>\tau+\varepsilon$, we consider the shifted cylinder
\begin{equation*}
\mathcal{C}_\varepsilon^{\tau'}:=\{(x,y-\varepsilon): (x,y)\in \mathcal{C}^{\tau'} \},
\end{equation*}
and the weighted space $W^{1,2}(\mathcal{C}_\epsilon^{\tau'},\rho)$, where
\begin{equation*}
\rho(x,y)=
\left\{\begin{array}{ccl}
y^{1-2s}&\text{if}& \,\,\,\,0<y<\tau'-\varepsilon,\\
(-y)^{1-2s}&\text{if}&-\varepsilon<y< 0.
\end{array}\right.
\end{equation*}
Further, let $\hat{u}_{k,\tau}\in W^{1,2}(\mathcal{C}_\varepsilon^{\tau'},\rho)$ defined by reflection as
\begin{equation*}
\hat{u}_{k,\tau}(x,y)=
\left\{\begin{array}{ccl}
u_{k,\tau}(x,y)&\text{if}& \,\,\,\,0<y<\tau'-\varepsilon,\\
u_{k,\tau}(x,-y)&\text{if}&-\varepsilon<y\leq 0,
\end{array}\right.
\end{equation*}
and note  that $\rho\in A_2(\mathcal{C}_\varepsilon^{\tau'})$, i.e.,
 	\begin{equation*}
 		\sup_{B\subset \mathcal{C}_\varepsilon^{\tau'}}\left(\frac{1}{|B|}\int_B \rho \dif X\right)\left(\frac{1}{|B|}\int_B \rho ^{-1} \dif X\right)^{}<+\infty,
 	\end{equation*}
 	for all squares $B\subset C^{\tau'}_\varepsilon$. 
 	
 	Let $L_r$ be the usual mollifier operator, i.e.,
 	\begin{equation*}
 		(L_r f)(x)=\frac{1}{r^{N+1}}\int_{\mathbb{R}^{N+1}}\omega\left(\frac{x-z}{r}\right) f(z)\dif z,
 	\end{equation*}
 	where $\omega:\mathbb{R}^{N+1}\to [0,+\infty)$ belongs to $C^\infty(\mathbb{R}^{N+1})$, $\mathrm{supp}\,\omega\subset \overline{B(0,1)}$, and $\int_{\mathbb{R}^{N+1}} \omega =1$. Since $\rho\in A_2(\mathcal{C}_\varepsilon^{\tau'})$ and since $\mathcal{C}_\varepsilon^{\tau'}$ is bounded and with Lipschitz boundary, it follows that for $f\in W^{1,2}(\mathcal{C}_\varepsilon^{\tau'},\rho)$, 
 	\begin{equation*}
 		L_rf\to f \qquad \text{ in } \quad W^{1,2}(D,y^{1-2s}),
 	\end{equation*}
 	for any $D\subset\subset \mathcal{C}_\varepsilon^{\tau'}$; see \cite{gol2009weighted}. Given that the support of $u_{k,\tau}$ is uniformly away from $\partial_L\mathcal{C}$, and that $u_{k,\tau}=0$ if $\tau<y$, it follows that
 	 	\begin{equation}\label{eq:vntrconv}
 		L_r\hat{u}_{k,\tau}\to u_{k,\tau} \qquad \text{ in } \quad W^{1,2}(\mathcal{C},y^{1-2s}),
 	\end{equation}
 	as $r\to\infty$. Note in addition, that for sufficiently large $r>0$, we have
 	 	 	\begin{equation*}
 		L_r\hat{u}_{k,\tau}\in C_c^{\infty}(\mathcal{C}_\Omega)\subset \mathscr{H}_{0,L}^{1,2}(\mathcal{C},y^{1-2s}).
 	\end{equation*}

\textit{Step 4:} In view of \eqref{eq:vnconv}, \eqref{eq:vntconv}, and \eqref{eq:vntrconv}, by appropriately selecting a sequence $\{(r_i,k_i,\tau_i)\}_{i=1}^\infty$, we observe
\begin{equation*}
	\tilde{u}_i:=L_{r_i}\hat{u}_{{k_i},{\tau_i}}\to u \qquad \text{ in } \quad W^{1,2}(\mathcal{C},y^{1-2s}),
\end{equation*}
as $i\to \infty$, so that $u\in \mathscr{H}_{0,L}^{1,2}(\mathcal{C},y^{1-2s})$. In particular, 
\begin{equation*}
	\tr \tilde{u}_i\to v \qquad \text{ in } L^2(\Omega),
\end{equation*}
and $v=\tr u$; the result is then proven.
\end{proof}

Next we can establish the well-posedness of the elliptic equation of interest.

\begin{theorem}\label{Thm:AbstractElliptic2} 
Assume that {\normalfont\ref{H1}}, {\normalfont\ref{H2}}, and {\normalfont\ref{H3}} holds true, and functions in $\mathscr{H}_{0,L}^{1,2}(\mathcal{C},w)$ satisfy a Poincar\'e inequality. For every $h\in L^2(\Omega,\tilde{w})$, the equation 
\begin{equation}\label{Eq:AbstractEq3}
(-\Delta)^{s(\cdot)}v=h,
\end{equation}
admits a unique solution $v\in\mathscr{Y}(\Omega,w)\subset L^2(\Omega,\tilde{w})$.  
\end{theorem}

\begin{proof}
The conclusion follows from the existence and uniqueness of solution to the same problem with right hand side in $\mathscr{Y}(\Omega,w)$ since we identify $ L^2(\Omega,\tilde{w})'\simeq  L^2(\Omega,\tilde{w})$, so that $ L^2(\Omega,\tilde{w})\subset \mathscr{Y}(\Omega,w)$ by means of $I':L^2(\Omega,\tilde{w})\to \mathscr{Y}(\Omega,w)'$.
\end{proof}

The result above can be refined in terms of regularity if in addition we observe \ref{H5}, and consider $\Omega=Q_N$. In this case, the injection $I$ is given as
\begin{equation*}
	I:\mathscr{Y}(Q_N,w)\to \mathbb{W}^{s(\cdot),2}(\mathrm{Q}_N,\tilde{w},w_1,\hdots,w_N),
\end{equation*} 
leading to our last theorem, whose proof is obtained as for Theorem \ref{Thm:AbstractElliptic2}. 
 
\begin{theorem}\label{Thm:AbstractElliptic3} 
In addition to the hypotheses of Theorem \ref{Thm:AbstractElliptic2}, consider $\Omega=\mathrm{Q}_N$ and assume that {\normalfont\ref{H5}} holds true.  Then, for every 
\begin{equation*}
	h\in \mathbb{W}^{s(\cdot),2}(\mathrm{Q}_N,\tilde{w},w_1,\hdots,w_N)',
\end{equation*} the equation \eqref{Eq:AbstractEq3}
admits a unique solution 
\begin{equation*}
	v\in\mathscr{Y}(\Omega,w)\subset \mathbb{W}^{s,2}(\mathrm{Q}_N,\tilde{w},w_1,\hdots,w_N).
\end{equation*}
\end{theorem}

The problem in the truncated cylinder $\mathcal{C}^\tau$ is treated identically, and Theorem \ref{Thm:AbstractElliptic2} and Theorem \ref{Thm:AbstractElliptic3} still hold true under the obvious changes.


\section{Conclusions and open questions}\label{Sect:Conclusions}
This paper continues the program initiated  in \cite{AnRa2019} and provides a rigorous definition 
of the variable order fractional Laplacian. The proposed theoretical framework enables solutions to Poisson equation 
on bounded Lipschitz domains $\Omega$. The techniques introduced in the paper are  new and 
none of the existing works applies to our setting. {However, the existing setting, where $s(\cdot)$ is a constant, can be 
recovered from our proofs as a special case.}

The following are open questions and topics for future research:
\begin{itemize}

\item The study of $-\Delta^{s(\cdot)}$ as regularizer in optimization problem, i.e., 
\begin{equation*}
	\min_u J(u)+\gamma R(u) \qquad \text{with}\quad R(u)=\langle(-\Delta)^{s(\cdot)}u, u \rangle_{\mathscr{Y}',\mathscr{Y}},
\end{equation*}
and the optimal selection of $s(\cdot)$ in a bilevel framework.

	\item The extension to more general settings of the Poincar\'e inequality type result presented in 
		Section~\ref{Sect:PoincareIneq}.
		
	\item The surjectivity of the new trace operator is still open (cf.~Remark~\ref{rem:surj}).

	\item We have introduced Sobolev spaces with $s(\cdot)$-dependent weights for the extension 
	problem and $s(\cdot)$-dependent differentiability for the space on $\Omega$. New approaches need 
	to be established to prove additional regularity of solutions to $(-\Delta)^{s(\cdot)}u=h$ in these Sobolev spaces.

	\item Extensions to parabolic, semilinear and obstacle type problems are of interest. 

	\item { The authors in \cite{AnRa2019} proposed a numerical method for the truncated problem. 
		 But the numerical analysis of this problem is completely open. Also, convergence of the
 		 truncated solution to the full solution is of interest as well.}

	\item Optimal control problems with variable order PDEs as constraints. 
\end{itemize}

\bibliographystyle{plain} 
\bibliography{References}


\end{document}